\documentclass[11pt]{amsart}
\usepackage{geometry}                
\usepackage{epsfig}
\usepackage{psfrag}
\usepackage{amssymb}
\usepackage{epstopdf}
\usepackage{color}
\usepackage{labelfig}
\newtheorem{theorem}{Theorem}[section]    
\newtheorem{lemma}[theorem]{Lemma}          
\newtheorem{proposition}[theorem]{Proposition}  
\newtheorem{claim}[theorem]{Claim}  
\newtheorem{corollary}[theorem]{Corollary} 

\theoremstyle{definition}
\newtheorem{definition}[theorem]{Definition}
\newtheorem{remark}[theorem]{Remark}  
\newtheorem{example}[theorem]{Example}    
\numberwithin{equation}{section}

\newcommand{\e}{\varepsilon}

\newcommand{\F}{\mathcal F_{ob}}
\newcommand{\cF}{\mathcal F_\xi }
\newcommand{\Int}{{\textrm{Int}}}
\newcommand{\sgn}{{\tt sgn}}
\newcommand{\Aut}{{\rm{Aut}}}
\newcommand{\Hom}{{\rm{Hom}}}
\newcommand{\id}{{\rm{id}}}
\newcommand{\RR}{\mathbb{R}}
\newcommand{\Z}{\mathbb{Z}}
\newcommand{\MCG}{\mathcal{MCG}}

\newcommand{\mF}{\mathcal{F} }
\newcommand{\mK}{\mathcal{K} }
\newcommand{\M}{\mathcal{M}}

\newcommand{\PD}{{\scriptstyle\rm{PD}}}

\begin{document}

\title{Open book foliations}

\author{Tetsuya Ito}
\address{Research Institute for Mathematical Sciences, Kyoto university, Kyoto, 606-8502, Japan}
\email{tetitoh@kurims.kyoto-u.ac.jp}
\urladdr{http://kurims.kyoto-u.ac.jp/~tetitoh/}

\author{Keiko Kawamuro}
\address{Department of Mathematics \\ 
The University of Iowa \\ Iowa City, IA 52240, USA}
\email{kawamuro@iowa.uiowa.edu}


\date{\today} 

\begin{abstract}
We study open book foliations on surfaces in $3$-manifolds, and 
give applications to contact geometry of dimension $3$. 
We prove a braid-theoretic formula of the self-linking number of transverse links, which reveals an unexpected link to the Johnson-Morita homomorphism in mapping class group theory. 
We also give an alternative combinatorial proof to the Bennequin-Eliashberg inequality. 
\end{abstract}

\maketitle

\tableofcontents


\section{Introduction}

In his seminal work \cite{Ben}, Bennequin shows that there is an ``exotic" contact structure, $\xi_{ot}$, on $S^{3}$ which is homotopic to the standard contact structure, $\xi_{std}$, as a $2$-plane field but {\em not} contactomorphic to $\xi_{std}$. 
In other words, $\xi_{std}$ is tight whereas $\xi_{ot}$ is overtwisted in contemporary terminology.  
In order to distinguish these contact structures he studies closed braids and characteristic foliations on their Seifert surfaces induced by the contact structures. 
Since then, Bennequin's method has been developed in two directions. 

One direction is the theory of characteristic foliations and convex surfaces:  
Eliashberg uses characteristic foliations to show the Bennequin-Eliashberg inequality for tight contact $3$-manifolds and generalizes the Bennequin inequality for the tight contact $3$-sphere in \cite{el2}. 
Characteristic foliations also play important roles in Eliashberg's  classification of overtwisted contact structures \cite{el1}.
In \cite{giroux1} Giroux extends characteristic foliation theory and initiates convex surface theory. 
This gives us cut-and-paste techniques to study contact structures and to classify tight contact structures for various $3$-manifolds.
See also Honda's work \cite{ho} on convex surface theory.

The other direction is the theory of braid foliations studied in a series of papers by Birman and Menasco \cite{BMiv, BMii, BMv, BMi, BMvi, BMiii, bm1}. One of its highest achievements is ``Markov theorem without stabilization'' which states that given two closed braid representatives of any link in $\RR^3$ can be transformed to each other in a very controlled manner \cite{bm1}. Moreover, Birman and Menasco apply the braid foliation to contact geometry and construct examples of transversely non-simple knots in the standard contact $3$-sphere: knots having the same topological type and the same self-linking number but not transversely isotopic~\cite{bm2}. 
Their examples are closed $3$-braids related by negative flypes. 
Analysis of a sequence of braid moves which relates one to the other (a Markov tower) reveals that the two closed braids represent distinct transverse links. See also \cite{bf} which is a concise survey article.

We establish a foundation of {\em open book foliations} which generalizes braid foliations.

Our starting point is a classical theorem that dates back to Alexander: every closed, oriented 3-manifold admits an open book decomposition.
An open book decomposition naturally induces a singular foliation on an embedded surface. When the foliation satisfies certain conditions, we call it an open book foliation on the surface. As shown in Theorem~\ref{prop:of}, any embedded surface can be isotoped to admit an open book foliation. 

An idea of open book foliation has been existing for some time. 
The project of this paper started from conversation between John Etnyre and the authors in the conference ``Braids in Seville'' in 2011. 
Etnyre pointed out that Bennequin's work \cite{Ben} suggests that characteristic foliations and open book foliations are essentially the ``same''. 
Also he and Ko Honda had discussed about generalization of braid foliations. 
In addition, readers may find a preliminary step toward open book foliations in Pavelescu's thesis \cite{Pav}.

However, a foundation of open book foliations in the general setting has not been fully developed in the literature, and an open book foliation has often been regarded as a special kind of characteristic foliations in contact geometry. 
In this paper we develop basics of open book foliations in topological and combinatorial way. 
It is important that open book foliation theory is independent of the theory of characteristic foliations. 
In fact in Remark~\ref{remark:rigid} we list items that highlight differences of the two foliations. 
A most notable difference is that open book foliations are more `rigid' than characteristic foliations. 
For instance, the Giroux cancellation lemma \cite{giroux2} for characteristic foliations does not apply to open book foliations.
However the two foliations have similar appearances as Etnyre pointed out to us.  
We prove {\em the structural stability theorem} (Theorem~\ref{identity theorem}) that states that the two foliations can be topologically conjugate to each other under certain conditions.

Hence via the Giroux-correspondence \cite{giroux3} open book foliation theory gives rise to a new technique to analyze general contact 3-manifolds, just like Bennequin's foliations and Birman-Menasco's braid foliations were used to study the standard tight contact $3$-sphere. 

Our first application of open book foliations to contact geometry is a self-linking number formula of an $n$-stranded braid, $b$, with respect to an open book $(S, \phi)$:
$$sl(\hat{b}, [\Sigma]) = -n + \widehat{\exp}(b) - \phi_*(a)\cdot[b] + c([\phi] , a)$$
The precise statement and definitions can be found in Theorem~\ref{theorem:sl-formula}. 
Interestingly the function $c([\phi], a)$ in the formula reveals an unexpected relationship between the self-linking number, an invariant in contact geometry, and the Johnson-Morita homomorphism in mapping class group theory. 
We discuss this in detail in \S\ref{sec:surface-one-boundary}.

Our formula generalizes Bennequin's self-linking number formula \cite{Ben} of a braid in the open book $(D^2, id)$, that is a usual  closed braid in $\RR^3$ around the $z$-axis. 
Bennequin's formula is:
$$sl(b) = -n + \exp(b)$$
where $\exp(b)$ (with no `hat' over $\exp$) is the exponent sum of a braid word representing $b$. 
When the page surface $S$ is an annulus Kawamuro and Pavelescu \cite{kp} show that: 
$$sl(b, [\Sigma])= -n + \exp(b) - \phi_*(a)\cdot [b]$$
Moreover if $S$ is planar Kawamuro \cite{k} shows that
$$sl(b, [\Sigma])= -n + \exp(b) - \phi_*(a)\cdot [b] + c'(\phi, a)$$
where the function $c'$ is a part of the function $c$ and the gap of $c$ and $c'$ is essentially the Johnson-Morita homomorphism mentioned above. 
We can see that the formula gets more complicated as the topology of $S$ gets complicated.

The self linking number is not merely an invariant of knots and links in contact manifolds. 
By the following celebrated Bennequin-Eliashberg inequality \cite{el2}, one can use the self-linking number to determine tightness or overtwistedness of a given contact structure:

\quad \\
\noindent
\textbf{Theorem~\ref{theorem:BEinequality}.} \cite{el2} 
{\em If a contact 3-manifold $(M,\xi)$ is tight, then for any null-homologous transverse link $L$ and its Seifert surface $\Sigma$, we have}: 
\[ sl(L,[\Sigma]) \leq -\chi(\Sigma) \]

Our second application of open book foliations to contact geometry is to give an alternative combinatorial proof to the above Bennequin-Eliashberg inequality. 
Because of its rigidity, an open book foliation is effective to visualize or construct surfaces like overtwisted discs. 
In fact we define a {\em transverse overtwisted disc} (Definition~\ref{def:trans-ot-disc}), a notion corresponding to an overtwisted disc in contact geometry, and we use it to reprove the Bennequin-Eliashberg inequality.

\subsection{Origins of open book foliation} 
\label{sec:origin}

{\quad }\\
In this section we briefly review braid foliations and characteristic foliations. 
We generalize braid foliations to open book foliations. 
On the other hand, many applications of open book foliations are derived from problems in characteristic foliation theory.

\subsubsection{Braid foliations} 

{\quad }\\
In Birman and Menasco's braid foliation theory \cite{BMiv, BMii, BMv, BMi, BMvi, BMiii, bm1}, braids are geometric objects.  
Let $A$ be an oriented unknot in $S^{3}$.  
We regard $S^{3}$ as $\RR^{3} \cup \{ \infty \}$ and identify $A$ with the union of the $z$-axis and the point $\{\infty\}$. 
As is well-known, $A$ is a fibered knot. 
With the cylindrical coordinates $(r, \theta, z)$ of $\RR^{3}$, the fiberation $\pi: S^{3}\setminus A \rightarrow S^{1}$ is given by the projection $(r, \theta, z) \mapsto \theta$.
An oriented link $L \subset S^3$ is called a {\em closed braid} with respect to $A$ (and $\pi$) if $L$ is disjoint from $A$ and positively transverse to each fiber $S_\theta=\pi^{-1}(\theta)$. 
In other words, $L$ winds around the $z$-axis in the positive direction.

Consider an incompressible Seifert surface $F$ of a braid $L$, or an essential closed surface $F \subset S^{3}\setminus L$. 
The intersection of $F$ and the fibers $\{S_{\theta}  \: | \: \theta \in S^1 \}$ induces a singular foliation $\mF$ on $F$. 
We can put $F$ in a position so that $\mF$ satisfies the following conditions: 
\begin{description}
\item[(i)] 
The $z$-axis pierces $F$ transversely in finitely many points around which the foliation $\mF$ is radial.

\item[(ii)] 
The leaves of $\mF$ along $\partial F$ are transverse to $\partial F$. 

\item[(iii)] 
All but finitely many fibers $S_\theta$ meet $F$ transversely.
Each exceptional fiber is tangent to $F$ at a single point.

\item[(iv)] 
All the tangencies of $F$ and fibers are saddles. 
\end{description}
This $\mF$ is called a {\em braid foliation} on the surface $F$. 
(Later in \S\ref{sec:def of OB} we borrow Birman and Menasco's axioms (i)--(iv) to define open book foliations.) 

The foliation $\mF$ encodes both topological and algebraic information of the closed braid $L$. 
For example, let $L$ be the closure of the braid word $\sigma_1$ in the Artin braid group $B_2$ and $F$ its Bennequin surface consisting of two discs and one positively twisted band.
The surface $F$ and its braid foliation $\mF$ are depicted in Figure~\ref{fig:braidfoliation}-(a). 
(The meaning of $\oplus$ will be made clear in Section \ref{subsubsec:sign}). 
We collapse the twisted band and the upper disc to get the trivial braid as in Figure~\ref{fig:braidfoliation}-(b). 
Algebraically this corresponds to {\em destabilization} of $\sigma_1$ and the topology of $\mF$ indicates that $L$ is destabilizable.
\begin{figure}[htbp]
\begin{center}
\SetLabels
\endSetLabels
\strut\AffixLabels{\includegraphics*[scale=0.5, width=100mm]{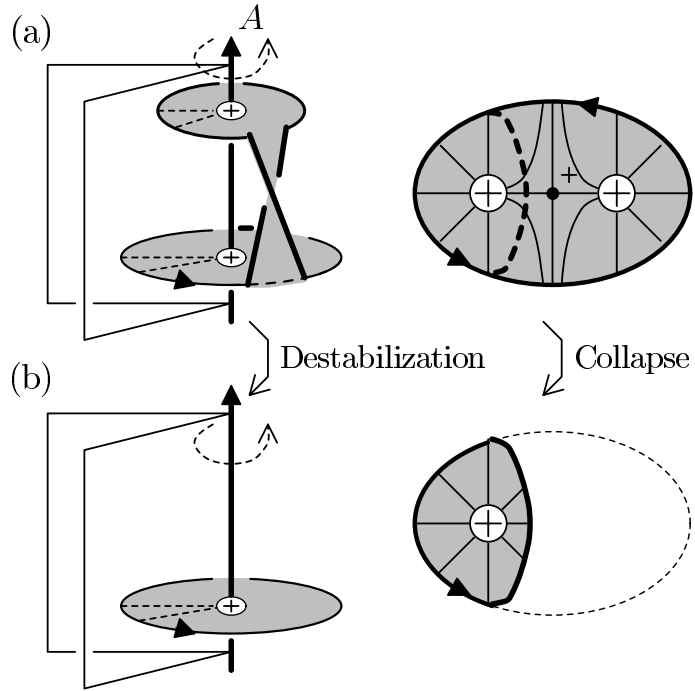}}
\caption{Example: Braid foliation and destabilization.}\label{fig:braidfoliation}
\end{center}
\end{figure}

In general if $\mF$ can be ``simplified'' then $L$ is also ``simplified''. 
Moreover, as the above example suggests, a simplification of $\mF$ can be understood as a certain braid operation.
Therefore, by studying $\mF$ one may find a sequence of braid operations to get the ``simplest'' braid representative of $L$.

Braid foliations have numerous applications to study of knots and links in $S^3$ \cite{BH, BMiv, BMii, BMv, BMi, BMvi, BMiii, bm1}. 
Moreover, via the correspondence \cite{Ben} between the transverse links in the standard contact $S^3$ and the closed braids around the $z$-axis, braid foliations are used to solve problems in contact geometry, in particular, detecting transversely non-simple links \cite{bm2, bm3, BW, LM}. 
(Here a topological link type $\mathcal L$ is called {\em transversely simple} if the transverse link representatives of $\mathcal L$ can be completely classified by an invariant called the self-linking number). 

In \cite{BMiii}, Birman and Menasco study the set of $3$-braids and prove that two closed $3$-braids representing topologically the same link are related to each other by the so called {\em flype} move. 
This is the key to their construction of transversely non-simple $3$-braid links \cite{bm2}. 
In \cite{BMv} they prove that every closed braid representative of the unknot can be deformed into the one-stranded braid by a sequence of exchange moves and destabilizations. 
Based on this, Birman and Wrinkle \cite{BW} give an alternative topological proof, first proven by Eliashberg and Fraser~\cite{ef}, that the unknot in $(S^{3},\xi_{std})$ is transversely simple. 

It should be pointed out that the braid foliation is not too difficult to see or illustrate once we understand how a surface is embedded. 
This contrasts strikingly with the flexibility of the characteristic foliations which we describe next.

\subsubsection{Characteristic foliations}

{\quad }\\
Let $(M, \xi)$ be a closed contact $3$-manifold. 
Let $F \subset M$ be an oriented embedded surface, usually either closed or with Legendrian boundary. (A convex surface with transverse boundary is established by Etnyre and Van Horn-Morris \cite[Section 2]{ev}.)
Integrating the vector field $\xi \cap TF$ on $F$ we get a singular foliation $\cF(F)$ on $F$ called the {\em characteristic foliation}.
If two contact structures induce the same characteristic foliation on $F$ then they are isotopic near $F$.

A surface $F$ is called {\em convex} if there exists a vector field $v$ whose flow preserves $\xi$ and is transverse to $F$.
The {\em dividing set} \cite{giroux1} on a convex surface $F$ is a multi-curve  defined by $\{ p \in F \: | \: v_p \subset \xi_p \}$.  
Giroux's flexibility theorem \cite{giroux1} \cite{ho} states that it is the isotopy type of a dividing set (not an individual characteristic foliation compatible with the dividing set) that encodes information of the contact structure near $F$.  
If two contact structures induce isotopic dividing sets on $F$ then they are isotopic near $F$.

In \cite{ho}, Honda introduces {\em bypass attachment} which allows us to modify dividing sets in controlled manner. 
With careful examination of dividing sets one can apply topological techniques such as gluing and cutting contact $3$-manifolds along convex surfaces.
This leads to various results in contact geometry. 
For example, Etnyre and Honda prove the non-existence of tight contact structures on a Poincar\'e homology sphere \cite{eh2}.
They also prove transverse non-simplicity of the $(2,3)$-cable of the $(2,3)$-torus knot by classifying its Legendrian representatives \cite{eh3}. 
Later, LaFountain and Menasco \cite{LM} establish Legendrian and transversal ``Markov theorem without stabilization'' for the above knot by using both braid foliation and convex surface techniques.

In practice, except for certain simple cases, it is not very easy to grasp the entire picture of a characteristic foliation and a dividing set.  
It is also not very clear how they change under isotopies of surfaces. 
In contrast, the structural stability theorem that we prove in \S\ref{sec:fromOB} allows us to visualize a characteristic foliation through an open book foliation.

\section{Basics of open book foliation}
\label{sec:OBfoliations}

In this section we define open book foliations and develop basic machinery by applying (sometimes with modifications) existing notions in braid foliation theory. 

Hence most of our definitions in this section can be found in Birman and Menasco's papers \cite{BMiv}-\cite{bm2}. 
We also cite Birman and  Finkelstein's paper \cite{bf} because it is a concise survey of braid foliation theory and conveniently contains all the basic notions we want to borrow.

\subsection{Definition of open book foliation}\label{sec:def of OB}  

{ \quad } \\
An {\em open book} $(S,\phi)$ is a compact surface $S$ with non-empty boundary $\partial S$ along with a diffeomorphism $\phi \in \Aut(S, \partial S)$ fixing the boundary pointwise.
Given an open book $(S,\phi)$ we define a closed oriented $3$-manifold $M=M_{(S,\phi)}$ by
\[ M_{(S,\phi)} = M_{\phi} \bigcup \left(\coprod_{|\partial S|} D^{2}\times S^{1} \right), \]
where $M_{\phi}$ denotes the mapping torus $S \times[0,1] \slash (x,1) \sim (\phi(x),0)$ and the solid tori are attached so that for each point $p \in \partial S$ the circle $\{p\} \times S^{1} \subset \partial M_\phi$ bounds a meridian disc of $D^2 \times S^1$. 
If a closed oriented manifold $M$ is homeomorphic to $M_{(S, \phi)}$ we say that $(S,\phi)$ is an {\em open book decomposition} of the manifold $M$.
For example, $M_{(D^2, \id)} \cong S^{3}$. 
The union of core circles of the attached solid tori, $B$, is called the {\em binding} of the open book.
Let $\pi: M \setminus B \rightarrow S^{1} = \RR \slash \Z$ denote the fibration. The fibers $\pi^{-1}(t)=:S_t$ where $t \in [0,1)$ are called the {\em pages} of the open book.

We say that an oriented link $L$ in $M_{(S,\phi)}$ is in {\em braid position} with respect to the open book $(S,\phi)$ if $L$ is disjoint from the binding and positively transverses each page $S_{t}$. 
This generalizes the familiar concept of braid position for $M_{(D^2, \id)} \simeq S^3$.

Let $F$ be an oriented, connected, compact surface smoothly embedded in $M_{(S, \phi)}$ whose boundary $\partial F$ (if it exists) is in braid position w.r.t. the open book $(S, \phi)$.

Consider the singular foliation $\mF=\mF(F)$ on $F$ induced by the the pages $\{S_t \ | \ t \in S^1\}$.
That is, $\mF$ is obtained by integrating the singular vector field $\{T_p S_{t} \cap T_pF\}_{p\in F}$ on $F$. 
We call each connected component of the integral curves a {\em leaf}. 
We may regard the leaves as $F \cap S_t$. 
By standard general position arguments (see \cite{Hirsch} for example) the surface $F$ can be perturbed while the braid isotopy class of $\partial F$ is fixed (if $\partial F$ is non-empty) so that $F$ satisfies the same conditions in \cite[p.23]{BMi}, namely 
\begin{description}
\item[($\mF$ i)] 
The binding $B$ pierces the surface $F$ transversely in finitely many points. 
Moreover, $p \in B \cap F$ if and only if there exists a disc neighborhood $N_{p} \subset \Int(F)$ of $p$ on which the foliation $\mF(N_p)$ is radial with the node $p$ (see the top sketches in Figure~\ref{fig:sign}). 
We call the singularity $p$ an {\em elliptic} point. 

\item[($\mF$ ii)] 
The leaves of $\mF$ along $\partial F$ are transverse to $\partial F$. 

\item[($\mF$ iii)] 
All but finitely many fibers $S_{t}$  intersect $F$ transversely.
Each exceptional fiber is tangent to $\Int(F)$ at a single point.
In particular, $\mF$ has no saddle-saddle connections.

\item[($\mF$ iv$'$)] 
The type of a tangency in {\bf ($\mF$ iii)} is saddle or local extremum. 
\end{description}

\begin{definition}
\cite[p.23]{BMi}
We say that a page $S_{t}$ is {\it regular} if $S_{t}$ intersects $F$ transversely and it is {\it singular} otherwise.
Similarly, a leaf $l$ of $\mF$ is called {\it regular} if $l$ does not contain a tangency point. 
\end{definition}

The arguments in \cite[p.272-273]{bf} imply the following:

\begin{proposition}\label{classification of regular leaves}
Since $\partial F$ is in braid position $($if $\partial F$ is non-empty$)$, no regular leaf of $\mF(F)$ has both of its endpoints on $\partial F$.
Hence, the regular leaves of $\mF$ are classified into the following three types:
 \begin{enumerate}
 \item[$a$-arc]: An arc where one of its endpoints lies on $B$ and the other lies on $\partial F$.
 \item[$b$-arc]: An arc whose endpoints both lie on $B$.
 \item[$c$-circle]: A simple closed curve.
 \end{enumerate} 
\end{proposition}

\begin{definition}
We say that the singular foliation $\mF(F)$ is an {\em open book foliation} if the above conditions {\bf ($\mF$ i, ii, iii)} and the following condition {\bf ($\mF$ iv)}, which is stronger than {\bf ($\mF$ iv$'$)}, are satisfied and we denote it by $\F(F)$.
\begin{description}
\item[($\mF$ iv)] 
All the tangencies of $F$ and fibers are of saddle type (see the bottom sketches of Figure~\ref{fig:sign}). 
We call them {\em hyperbolic} points.
\end{description}
\end{definition}

\begin{remark}
Here we list differences between the braid foliation and the open book foliation. 
\begin{enumerate}
\item 
For braid foliations the ambient manifold $M$ is $S^3$, whereas for  open book foliations $M$ can be any closed oriented $3$-manifold. 

\item
In braid foliation theory each regular leaf $l \subset S_{t}$ is required to be {\em essential} in $S_{t}\setminus (S_{t} \cap \partial F)$ \cite[Theorem 1.1]{bf}. 
In open book foliation theory we relax this restriction, so   a regular leaf can be inessential, i.e., $F$ can be compressible.

We do this for the following reasons: 
First, we prefer to establish basics of open book foliations under less restrictive conditions. 
Second, characteristic foliations, which share common properties with open book foliations, also contain inessential circles. 
Third, in some cases it is more convenient and natural to allow inessential leaves: 
For example, we will see in Proposition \ref{prop:no-c-circle}, one can remove c-circles at the cost of introducing inessential leaves.
(In \cite{ik} we study open book foliations all of whose  b-arc leaves are essential and give several applications to topology of 3-manifolds.) 
\end{enumerate}
\end{remark}

The open book foliation is intrinsic in the following sense:

\begin{theorem}
\label{prop:of}
If {\bf ($\mF$ i, ii, iii, iv')} are satisfied then {\bf ($\mF$ iv)} holds. 
Namely, with an ambient isotopy $($that fixes $\partial F$ if it exists$)$, every surface $F$ admits an open book foliation $\F(F)$. 
\end{theorem}

We prove Theorem~\ref{prop:of} in \S~\ref{subsec:pf of thm}. 
At a glance this theorem is similar to \cite[Lemma~2]{BMi}. 
However we allow our pages to be of type $S_{g, r}$ (rather than  $D^2$) and moreover we allow $F$ to be compressible. 
As a result Birman and Menasco's proof (that is a refined argument of Bennequin's \cite{Ben} with much more details) does not apply. 
For the same reason, Roussarie-Thurston's argument \cite{T1} does not work either. 

As a byproduct of the proof of Theorem~\ref{prop:of} we obtain:

\begin{proposition}
\label{prop:no-c-circle}
Given an open book foliation $\F(F)$ we can perturb $F$  $($fixing $\partial F$ if it exists$)$ so that the new $\F(F)$ contains no $c$-circles. 
\end{proposition} 

We prove Proposition~\ref{prop:no-c-circle} also in \S~\ref{subsec:pf of thm}. 
This is a useful proposition that allows us to convert an open book foliation into Morse-Smale type. 
In this paper we use Proposition~\ref{prop:no-c-circle} many times, including a new proof of the Bennequin-Eliashberg inequality.

\subsubsection{Signs of singularities, describing arcs, and orientation of leaves.}\label{subsubsec:sign}

\begin{definition} \cite[p.19]{Ben} \cite[p.280]{bf} 
We say that an elliptic singularity $p$ is {\em positive} ({\em negative}) if the binding $B$ is positively (negatively) transverse to $F$ at $p$.
The sign of the hyperbolic singularity $p$ is {\em positive} ({\em negative}) whether the orientation of the tangent plane $T_p F$ does (does not) coincide with the orientation of $T_p S_t$. 

See Figure \ref{fig:sign}, where we describe an elliptic point by a hollowed circle with its sign inside, a hyperbolic point by a black dot with the sign indicated nearby, and positive normals  $\vec n_F$ to $F$ by dashed arrows. 
\begin{figure}[htbp]
 \begin{center}
\includegraphics[width=130mm]{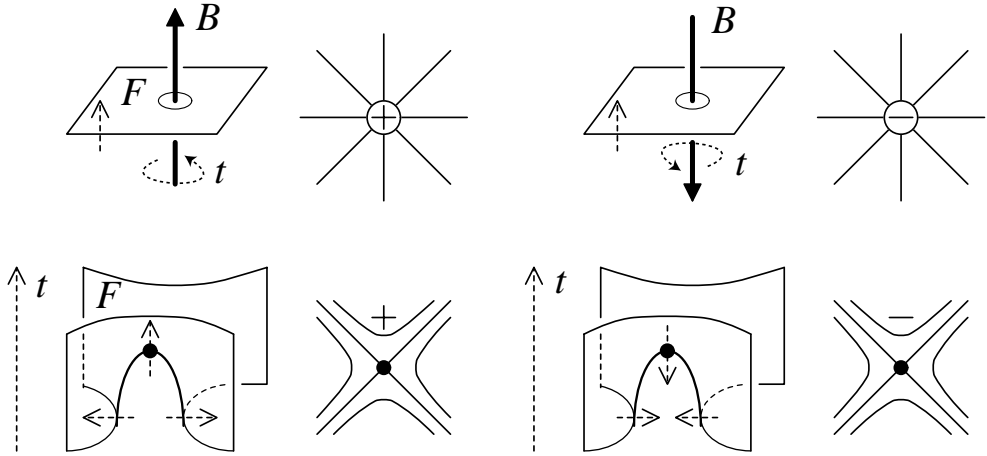}
 \end{center}
 \caption{Signs of singularities and normal vectors $\vec n_F$.}
  \label{fig:sign}
\end{figure}
\end{definition} 

With this definition, we observe that:

\begin{claim}\label{sign-observation}
The elliptic point at the end of every $a$-arc is positive, and
the endpoints of every $b$-arc have opposite signs.
\end{claim}

\begin{definition}[Describing arc]
\label{defn:describing_arc}
Consider a saddle shape subsurface of $F$ whose leaves $l_1$ and $l_2$ (possibly $l_1=l_2$) as in Figure~\ref{fig:hyperbolic} are sitting on a page $S_t$.  
As $t$ increases (the page moves up) the leaves converge along a properly embedded arc $\gamma \subset S_t$ (dashed in Figure~\ref{fig:hyperbolic}) joining $l_{1}$ and $l_{2}$ and switch configuration. 
See the passage in Figure~\ref{fig:hyperbolic}. 
\begin{figure}[htbp]
\begin{center}
\SetLabels
(.1*.78) $\gamma$\\
(.36*.82) $F$\\
(.52*.55) $l_2$\\
(.05*.86) $l_1$\\
(.18*.86)   $l_2$\\
(.46*.59) $\gamma$\\
(.36*.55) $l_1$\\
(.58*.55)   $S_t$\\
\endSetLabels
\strut\AffixLabels{\includegraphics*[width=100mm]{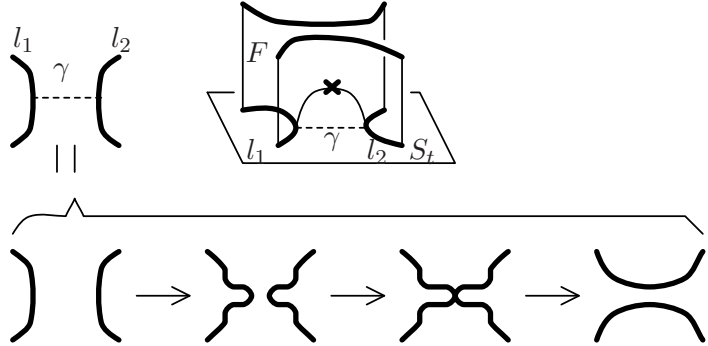}}
 \caption{The describing arc (dashed) for a hyperbolic singularity.}
 \label{fig:hyperbolic}
\end{center}
\end{figure}
We call $\gamma$ the {\em describing arc} of the hyperbolic singularity. 
Up to isotopy, $\gamma$ is uniquely determined. 
We also often put the sign of a hyperbolic point near its describing arc (see Figure \ref{fig:movie}). 
\end{definition}

\begin{definition}\label{def of e}
We denote the number of positive (resp. negative) elliptic points of $\F(F)$ by $e_{+}=e_{+}(\F(F))$ (resp. $e_{-}=e_{-}(\F(F))$). Similarly, the number of positive (resp. negative) hyperbolic points is denoted by $h_{+}=h_{+}(\F(F))$ (resp. $h_{-}=h_{-}(\F(F))$). 
\end{definition}

\begin{proposition}\label{poincare-hopf}
The Euler characteristic of the surface $F$ has
$$\chi(F) = (e_+ + e_-) - (h_+ + h_-).$$
\end{proposition}

To prove Proposition~\ref{poincare-hopf}, we define {\em orientations} of leaves:

\begin{definition}[Orientation of leaves]\label{def ori of leaf}
Both the surface $F$ and the ambient manifold $M$ are oriented so that the positive normal $\vec n_F$ of $F$ (in this paper we indicate $\vec n_F$ by dashed arrows like in Figure~\ref{fig:sign}) is canonically defined. 
We orient each leaf of $\F(F)$, for both regular and singular, so that if we would stand up on the positive side of $F$ and walk along a leaf, the positive side of the intersecting page $S_t$ of the open book would be on our left. 
In other words, at a non-singular point $p$ on a leaf $l \subset (S_t \cap F)$ let $\vec n_S$ be a positive normal to $S_t$ then $X_{ob}:= \vec n_S \times \vec n_F$ is a positive tangent to $l$. 
As a result, positive/negative elliptic points are sources/sinks of the vector field $X_{ob}$. 
\end{definition}

\begin{proof}[Proof of Proposition~\ref{poincare-hopf}] 
The orientation of the leaves gives a vector field $X_{ob}$ on $F$. 
By the axiom ($\mF$ iv) any singularity of $\F(F)$ is either elliptic or hyperbolic. The statement follows from the Poincar\'e-Hopf theorem. 
\end{proof}

\subsubsection{Proofs of Theorem~\ref{prop:of} and Proposition~\ref{prop:no-c-circle}}\label{subsec:pf of thm}

{ \quad } \\
Since we do not assume incompressibility of the surface $F$, we {\em cannot} directly apply Roussarie-Thurston's general position theorem~\cite[Theorem~4]{T1} or the proof of a corresponding result in braid foliation theory \cite[Lemma 2]{BMi} in order to remove all the local extrema from a foliation satisfying {\bf($\mF$  i, ii, iii, iv')}. 
Instead, we use a trick which we call a finger move. 

\begin{proof}[Proof of Theorem~\ref{prop:of}]

Let $F$ be a surface in a general position such that the singular foliation $\mF=\mF(F)$ satisfies {\bf($\mF$  i, ii, iii, iv')}. We show that we can isotope $F$ so that {\bf($\mF$ iv)} is satisfied.

Let $p$ be a local extremal point on the page $S_{t}$. 
We will replace $p$ with a pair of elliptic points and one hyperbolic point by the following isotopy, which we call a {\em finger move}. Repeating finger moves we can get rid of all the local extrema, i.e., {\bf($\mF$ iv)} is satisfied.

Choose an arc $\gamma$ in $S_{t}$ that joins $p$ and a binding component $B$. See Figure~\ref{Finger move}. 
If $\gamma$ intersects other regular leaves of $\mF(F)$, by small local perturbation we make the intersections transverse. 
Take a small $3$-ball neighborhood $N$ of $\gamma$ (dashed ellipses). 
We may assume that $N$ contains no singularities of $\mF$ other than $p$. 
Push a neighborhood of $p$ along $\gamma$ so that no  changes occur outside the region $N$. 
See the passage in Figure~\ref{Finger move}-(a):  
\begin{figure}[htbp]
\SetLabels
(.3*.3) $N$\\
(.85*.27) $N$\\
(.87*.5) $S_t$\\
(.44*.18) $S_t$\\
(.93*.18) $S_t$\\
\endSetLabels
\strut\AffixLabels{\includegraphics*[width=130mm]{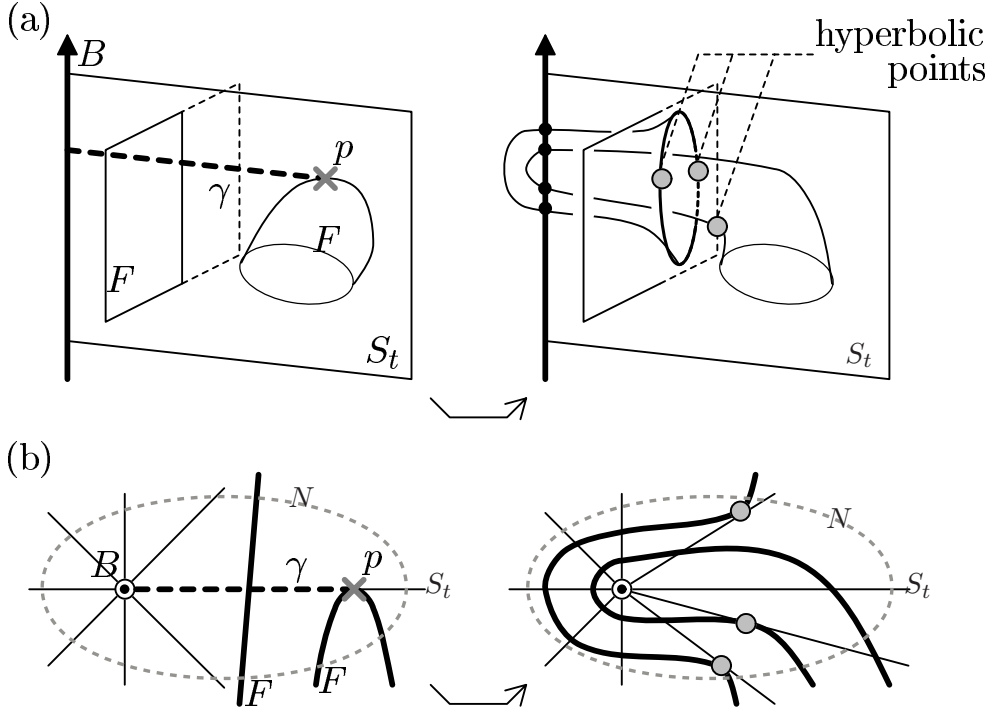}}
\caption{A finger move supported on a small neighborhood $N$ of $\gamma$.}\label{Finger move}
\end{figure}
Call this isotopy a {\em finger move} supported on $N$. 
Figure~\ref{Finger move}-(b) illustrates this finger move viewed from `above' the binding component $B$.

The finger move removes $p$ and introduces new elliptic (black dots in Figure~\ref{Finger move}) and hyperbolic (gray dots) singularities to $\mF$. 
But since the finger move is supported on $N$ no new local extrema are introduced.  
More precisely;  
if a positive normal to $S_t$ agrees (resp. disagrees) with a positive normal to $F$ at $p$, then the finger move introduces one negative (resp. positive) hyperbolic point and a pair of $\pm$ elliptic points. See the top passage in Figure~\ref{foliation finger move}.
For other part of $F$ that is involved in the finger move, a pair of $\pm$ elliptic points and a pair of $\pm$ hyperbolic points are inserted. See the bottom passage of Figure~\ref{foliation finger move}.
\begin{figure}[htbp]
\begin{center}
\includegraphics[width=100mm]{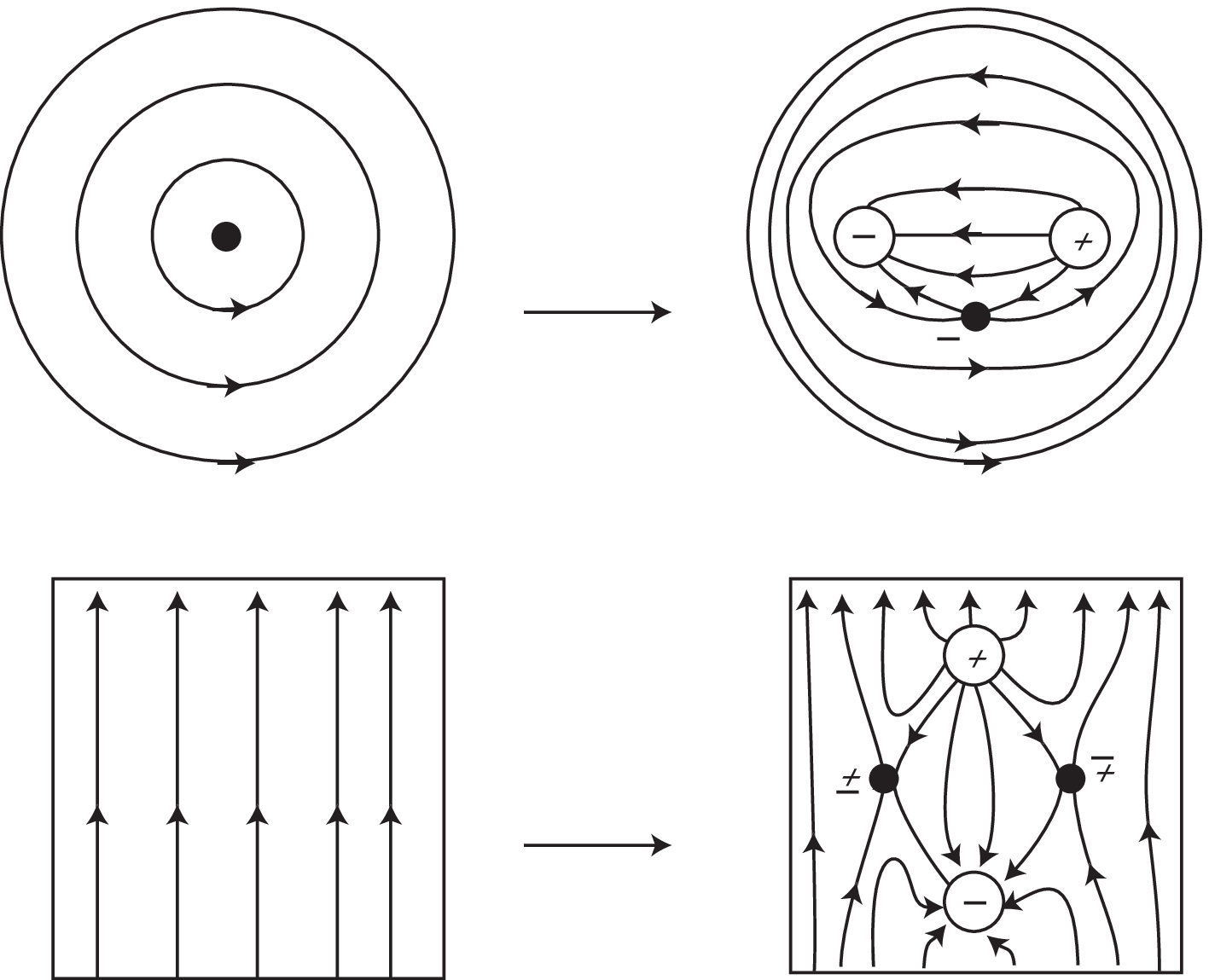}
\caption{(Top) Foliation change by a finger move near a local maximum. (Bottom) Non-singular region involved in a finger move.}
\label{foliation finger move}
\end{center}
\end{figure}
\end{proof}

\begin{proof}[Proof of Proposition~\ref{prop:no-c-circle}] 

Let $\F(F)$ be an open book foliation containing c-circles. 
For a c-circle $c$ there exists a maximal annulus $c \subset A_c \subset F$ whose interior is foliated only by c-circles and whose boundary components are singular leaves. 
Let us call $A_c$ a {\em c-circle annulus}. 
The number of c-circle annuli in $\F(F)$ is finite since the number of singularities of $\F(F)$ is finite. 

In the following, applying finger moves introduced in the proof of Theorem~\ref{prop:of} we will eliminate all the c-circle annuli. Recall that a finger move does not introduce new c-circles.

Let $A \subset F$ be a c-circle annulus whose interior consists of a smooth family of c-circles $\{c_t \subset S_t\}$ and let $c_{t_i} \subset S_{t_i} \cap \partial A$ ($i=0, 1$) denote the limit circles of the family. 
There is no restriction on the way that $A$ may wind around the binding components. 
Each limit circle $c_{t_i}$ has one (or two) hyperbolic point(s). (In the latter case the two points must be identical due to the condition {\bf($\mF$ iii)} and the limit circle is immersed like the singular leaf in a cc-pants as in Figure~\ref{region}.)

Since the open book foliation $\F(F)$ contains only finitely many hyperbolic points, there exists a finite family of disjoint smooth arcs and points 
$$\{\mbox{ arc } \alpha_i \subset A,\  \mbox{ point } p_i \in B \ |\  i=1,\cdots,k\},$$ 
where $B$ is the set of binding components, 
such that 
\begin{itemize}
\item
every c-circle of $A$ intersects at least one of the arcs $\alpha_i$,
\item
all the intersections of $\alpha_i$ and c-circles are transverse,
\item 
for each $\alpha_i$ there exists a smooth family of arcs $\lambda^i_t \subset S_t$ from the point $\alpha_i \cap S_t$ to $p_i$ that avoids hyperbolic points of $\F(F)$ and is never tangent to leaves of $\F(F)$. See Figure~\ref{fig:c-circle annulus}. (It is  convenient to imagine a triangle $\Delta_i = \{\lambda^i_t\}$ with the bottom edge $\alpha_i$ and the top vertex $p_i$.)
\end{itemize}
\begin{figure}[htbp]
\begin{center}
\SetLabels
(.85*.5) $A$\\
(-.07*.47) $p_i$\\
(.84*.25) $\alpha_1$\\
(.64*.5) $\alpha_i$\\
(.65*.9) $\alpha_k$\\
(.25*.45) $\Delta_i$\\
(.5*.02) $c_{t_0}$\\
(.5*1) $c_{t_1}$\\
\endSetLabels
\strut\AffixLabels{\includegraphics*[width=30mm]{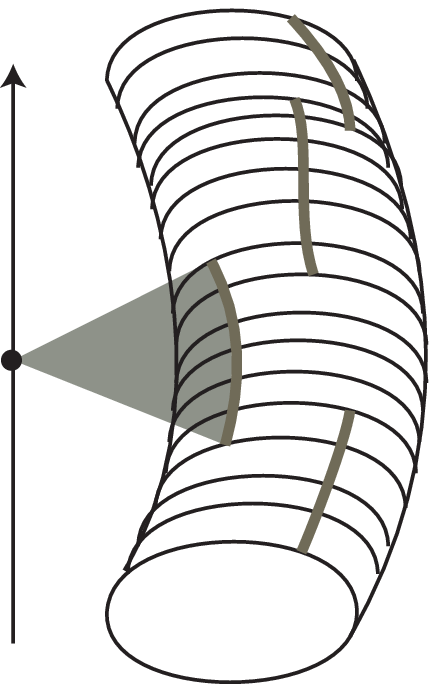}}
\caption{Arc $\alpha_i$, point $p_i$ and triangle $\Delta_i$.}\label{fig:c-circle annulus} 
\end{center}
\end{figure}

We apply a finger move (see the proof of Theorem~\ref{prop:of}) along the triangle $\Delta_i$.
The open book foliation locally changes as in the bottom passage of Figure~\ref{foliation finger move} in a neighborhood of $\alpha_i$. 
Then all the c-circles through $\alpha_i$ disappear. 
Repeat finger moves along all $\Delta_1, \cdots, \Delta_k$. 
As a consequence all the c-circles of $A$ disappear. 
Note that the finger moves may introduce new singularities even away from $A$ if some $\Delta_i$ intersect other parts of the surface $F$. 
We apply this procedure to every c-circle annulus. 
\end{proof}

\subsubsection{Regions}

\begin{definition}\cite[p.30]{BMi} 
Recall the three types of regular leaves: Type $a, b$ and $c$ (Proposition~\ref{classification of regular leaves}). 
The hyperbolic points in $\F(F)$ are classified into six types, according to types of nearby regular leaves: {\em Type aa, ab, bb, ac, bc,} and {\em cc} as depicted in Figure ~\ref{region}.
We call such model regions {\em aa-tile, ab-tile, bb-tile, ac-annulus, bc-annulus, cc-pants}, respectively. 
(Note that $ac$-annuli do {\em not} exist in braid foliation theorey \cite[p.279]{bf}.)
\begin{figure}[htbp]
\begin{center}
\SetLabels
(0*.92) $\partial F$\\
(.3*.58) $\partial F$\\
(.6*.58) $\partial F$\\
(.3*.06) $\partial F$\\
(0.15*0.55)  aa-tile\\
(0.5*0.55)    ab-tile\\
(0.84*0.55)  bb-tile\\
(0.15*0.04)  ac-annulus\\
(0.5*0.04)    bc-annulus\\
(0.84*0.04)  cc-pants\\
\endSetLabels
\strut\AffixLabels{\includegraphics*[scale=0.5, width=90mm]{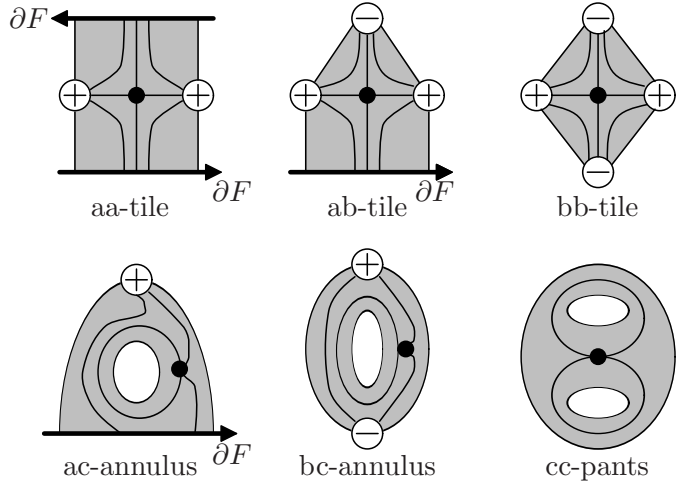}}
\caption{The six types of regions.}\label{region}
\end{center}
\end{figure}

For each region, the sign of the hyperbolic point can be either $+1$ or $-1$, but the signs of the elliptic points are determined as depicted in Figure~\ref{region} due to  Claim~\ref{sign-observation}. 
For $ac$- and $bc$-annuli, the hyperbolic points can be on the left part of the annuli.  
The interior of a region is embedded in $F$ as a disc, an annulus, or a pair of pants.
\end{definition}

\begin{definition}[Degenerate regions]
If a region $R$ is of type $aa$, $ac$, $bc$ or $cc$ some parts of $\partial R$ are possibly identified in $F$. In such case we say that $R$ is {\em degenerate}.  
For example, in Figure~\ref{fig:degenerate_region}-(1) two boundary a-arcs of an aa-tile are identified, and in (2) the two boundary b-arcs of a bc-annulus are identified (we have already seen this in Figure~\ref{foliation finger move}).

On the other hand, a region like in Figure~\ref{fig:degenerate_region}-(3), where two ends of the singular leaf lie on the same positive elliptic point, does not exist. 
This is because around an elliptic point all the leaves (both regular and singular) sit  on distinct pages. 
\end{definition}
\begin{figure}[htbp]
\begin{center}
\SetLabels
(0*0.9)  (1)\\
(0.34*0.9)  (2)\\
(0.7*0.9)  (3)\\
  \endSetLabels
\strut\AffixLabels{\includegraphics*[scale=0.5, width=90mm]{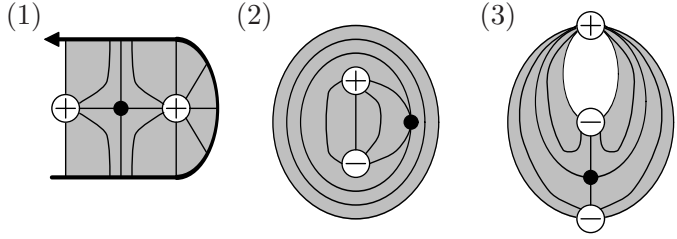}}
 \caption{(1) A degenerate aa-tile. (2) A degenerate bc-annulus. (3) Non-existing region.}
\label{fig:degenerate_region}
\end{center}
\end{figure}
We study degenerate regions in \cite{ik2}.

The next proposition shows one of the useful combinatorial features of open book foliations. It is originally a theorem in braid foliation theory. 

\begin{proposition}[Region decomposition]  
\label{prop:region}
\cite[Theorem 1.2]{bf} 
If $\F(F)$ contains a hyperbolic point, the surface $F$ is decomposed into a union of model regions whose interiors are disjoint. 
\end{proposition}

We omit a proof and refer the readers to the proof of \cite[Theorem 1.2]{bf}.  

The decomposition is called a {\em region decomposition} of $F$. 
It describes how $F$ is embedded in $M_{(S, \phi)}$.
If $\F(F)$ has no $c$-circles then the region decomposition gives a cellular decomposition of $F$.

\subsubsection{The graph $G_{--}$}

\begin{definition}
The two flow lines, induced by the orientation vector field $X_{ob}$ on $F$ (Definition~\ref{def ori of leaf}), approaching to (resp. departing from) the hyperbolic point in an $aa$-, $ab$-, or $bb$-tile is called {\em stable} (resp. {\em unstable}) {\em separatrices}.
\end{definition}

\begin{definition}\label{def:negativity-graph}
\cite[p.471]{bm1}
The {\em graph $G_{--}$} is a graph embedded in $F$. 
The edges of $G_{--}$ are the unstable separatrices for negative hyperbolic points in $aa$-, $ab$- and $bb$-tiles. 
See Figure \ref{fig:neggraph}. 
We regard the negative hyperbolic points as part of the edges. 
The vertices of $G_{--}$ are the negative elliptic points in $ab$- and $bb$-tiles and the end points of the edges of $G_{--}$ that lie on $\partial F$, called the {\em fake} vertices.
\begin{figure}[htbp]
 \begin{center}
\SetLabels
(0.03*0.95) aa-tile\\
(0.3*0.95) ab-tile\\
(0.58*0.95) bb-tile\\
(0.90*0.1) \Large : $G_{--}$\\
(0.9*0.29) : Fake vertex\\
  \endSetLabels
\strut\AffixLabels{\includegraphics*[scale=0.5, width=120mm]{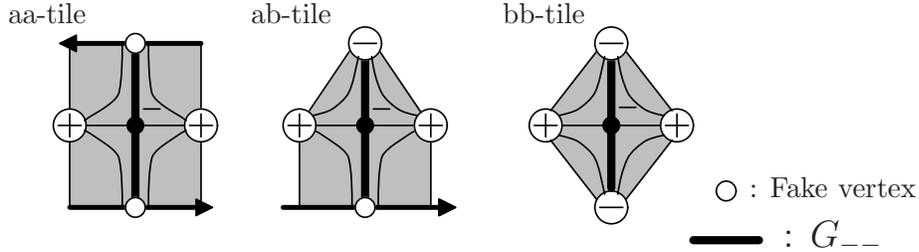}}
 \caption{The graph $G_{--}$.}\label{fig:neggraph}
  \end{center}
\end{figure}

In the same way we can define $G_{++}$ the graph consists of positive elliptic points and stable separatrices of positive hyperbolic points.  
\end{definition}

\begin{remark}
The origin of the above definition is the graphs $G_{\pm \pm}, G_{\pm \mp}$ \cite[p.314]{bf} \cite[p.471]{bm1} in braid foliation theory. 
In convex surface theory our $G_{--}$ corresponds to a sub-graph of the Giroux graph \cite[p.646]{giroux2}. 
\end{remark}

We will use the graphs $G_{--}$ and $G_{++}$ to define a transverse overtwisted disc and to give an alternative proof to the Bennequin-Eliashberg inequality in \S\ref{sec:OT-disc}. 

\subsubsection{Movie presentation}\label{sec:movie_presentation}

{\qquad}\\
A useful tool for expressing how the surface, $F$, is embedded in $M_{(S, \phi)}$, {\em movie presentations}, can be borrowed from braid foliation theory, see \cite[Fig 8]{BMi}. 
Using a movie presentation allows us to grasp the whole picture of $\F(F)$.

Let $\{S_{t_{i}}\}_{i=1,\ldots,k}$ be the set of singular pages of $\F(F)$, where $0<t_{1}<t_{2}< \cdots < t_{k}<1$. 
Consider the family $\{(S_t, F\cap S_t) \ | \ t \in [0, 1]\}$ of slices of $F$ by the pages $S_t$.  
For $s, t \in (t_i, t_{i+1})$ the slices $(S_s, F\cap S_s)$ and $(S_t, F\cap S_t)$ are isotopic, and the isotopy type of $(S_{t},F\cap S_{t})$ changes only when $t=t_i$. 
The describing arcs (Definition~\ref{defn:describing_arc}) encode all the information of the configuration changes. 

Choose $s_0=0,s_k=1$ and $s_i \in (t_i, t_{i+1})$. 
Consider the slices $\{(S_{s_i}, F \cap S_{s_i}) \ |\ i=0,\ldots,k\}$. 
These are the slices on which we may place describing arcs. 
The describing arc for the singularity on $S_{t_i}$ is found on $S_{s_{i-1}}$. 
The above observation shows that those are the the slices that determine the embedding of $F$ and the open book foliation $\F(F)$ up to isotopy. 
The slice $(S_{1}, F \cap S_{1})$ is identified with the slice $(S_{0}, F\cap S_{0})$ under the monodromy $\phi$. 
We call this family of slices with describing arcs a {\em movie presentation} of $\F(F)$. 

We will often use part of a movie presentation to express a local picture of a surface. 
Also, for the reader's convenience, some movie presentations may contain singular slices $(S_{t_i}, F\cap S_{t_i})$ like in Figure~\ref{fig:of_exam_2}.

\subsubsection{Examples of open book foliations} 

\begin{example}\label{the simplest example}
First we consider the simplest open book $(D^{2},id)$ which supports the standard tight contact structure on $S^3$. 
This is the case that Birman and Menasco studied in their braid foliation theory. 
Consider a $2$-sphere $F$ embedded as shown in the left sketch of Figure~\ref{fig:of_exam_1}. 
\begin{figure}[htbp]
 \begin{center}
\SetLabels
  \endSetLabels
\strut\AffixLabels{\includegraphics*[scale=0.5, width=80mm]{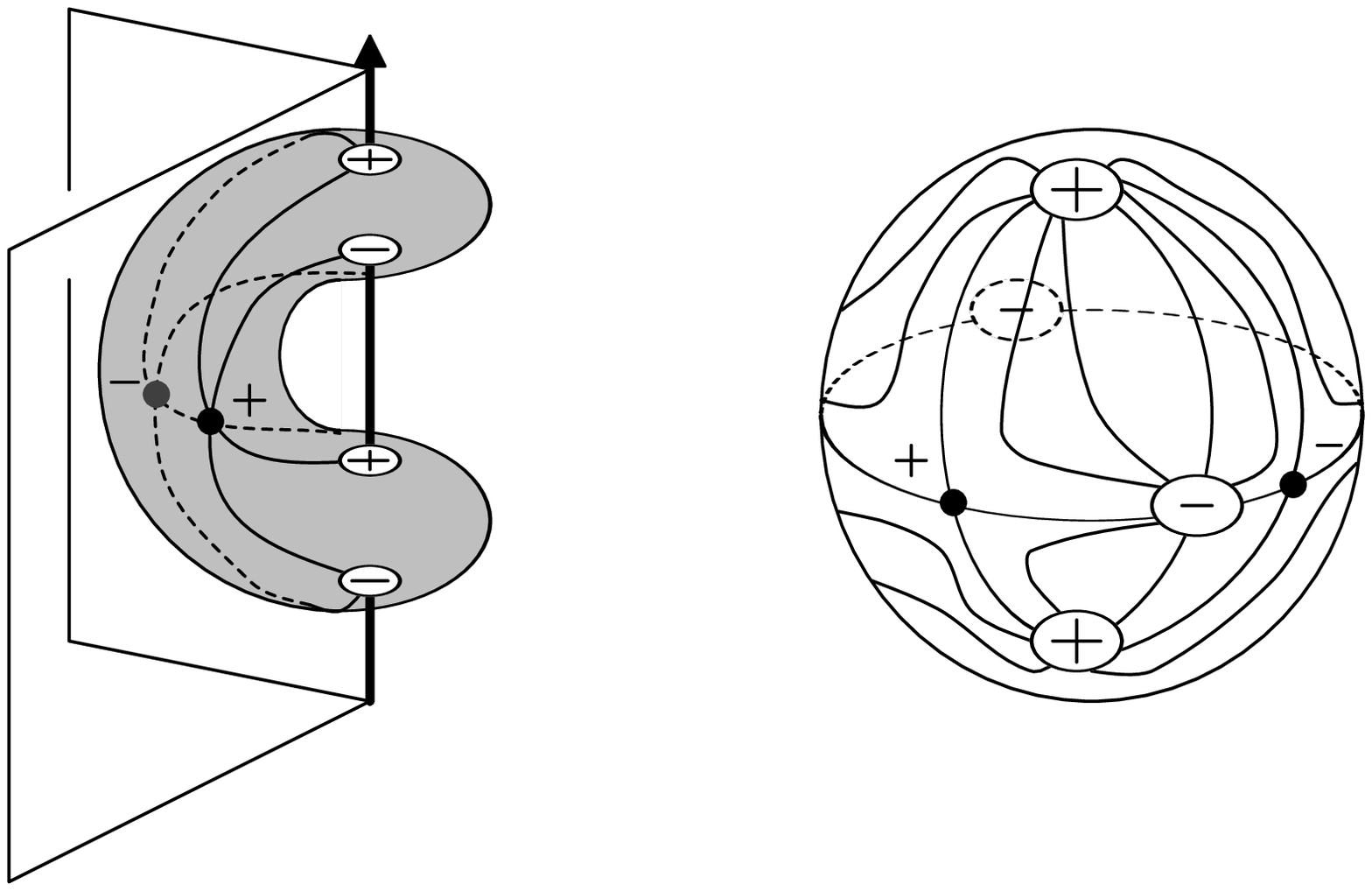}}
 \caption{Example~\ref{the simplest example}.}\label{fig:of_exam_1}
  \end{center}
\end{figure}

Since $F$ intersects the binding in four points, the open book foliation $\F(F)$ has four elliptic points, two positive and two negative. It also has two hyperbolic points of opposite signs where $F$ is tangent to pages of the open book (we may assume that the hyperbolic points lie on pages $S_{1/4}$ and $S_{3/4}$). 
The right sketch of Figure~\ref{fig:of_exam_1} depicts the whole picture of $\F(F)$ and Figure~\ref{fig:movie} depicts a movie presentation of $\F(F)$, where the dashed arrows indicate positive normals to $F$. 
Note that the open book foliation $\F(F)$ contains inessential $b$-arcs so, strictly speaking, this foliation is not treated in braid foliation theory.  
\begin{figure}[htbp]
 \begin{center}
 \SetLabels
  (0.12*0.3) $S_{0}$\\
  (0.44*0.3) $S_{\frac{1}{2}}$\\
  (0.74*0.3) $S_{1}$\\
  (0.15*0.7) $-$\\
  (0.51*0.62) $+$\\
\endSetLabels
\strut\AffixLabels{\includegraphics*[scale=0.5, width=90mm]{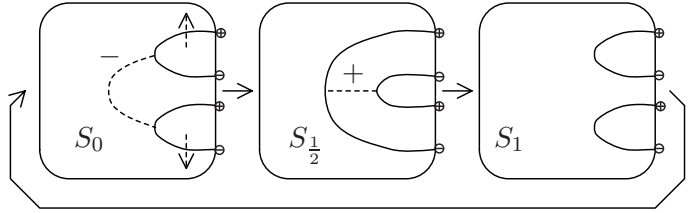}}
 \caption{A movie presentation of Example~\ref{the simplest example}.}
\label{fig:movie}
  \end{center}
\end{figure}
\end{example}

\begin{example}\label{informative example}
Next we study a more informative example.
Consider the open book $(S, \phi):= (A,T_{A}^{-1})$ where $A$ denotes an annulus and $T_{A}$ the right-handed Dehn twist along a core circle of $A$.
The ambient manifold is again $S^{3}$. 
However in this case the binding is a negative Hopf link and $(A,T_{A}^{-1})$ supports an overtwisted contact structure.

In order to visualize an overtwisted disc, $D$, we cut the complement of the binding $S^3 \setminus B$ along the page $S_{0}$. The resulting manifold $\simeq S \times [0,1]$ and each page $S_{t}$ is naturally identified with $S \times \{t\}$. 
The disc $D \subset M_{(S, \phi)}$ is also cut out along $D \cap S_0$ and becomes a properly embedded surface, $D'$,  in $S \times [0,1]$ such that $D'\cap (S\times \{0\})= \phi(D' \cap (S\times \{1\}))$ and $D' \cap (\partial S \times [0,1]) = (D \cap \partial S_0)\times [0,1]$. 
The left sketch in Figure~\ref{fig:of_exam_2} shows how $D'$ is embedded in $S\times [0, 1]$. 
\begin{figure}[htbp]
\begin{center}
\SetLabels
\endSetLabels
\strut\AffixLabels{\includegraphics*[scale=0.5, width=90mm]{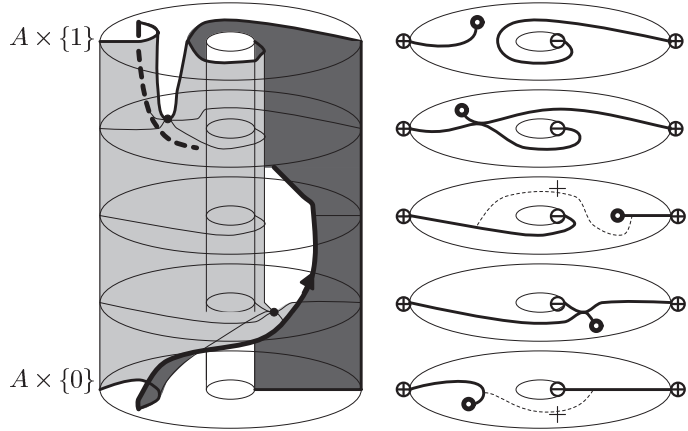}}
\caption{Example~\ref{informative example}. An overtwisted disc in an annulus open book $(A,T_{A}^{-1})$.}
\label{fig:of_exam_2}
\end{center}
\end{figure}
The sketch on the right depicts a movie presentation of $\F(D)$.
(For convenience, as we note in \S~\ref{sec:movie_presentation}, redundant slices that contain hyperbolic points are added in the 2nd and 4th rows.) 
We see that the multi-curve in the top annulus is identified with the multi-curve in the bottom annulus under the monodromy  $T_{A}^{-1}$. 
The movie presentation also shows that the open book foliation $\F(D)$ contains two positive hyperbolic points, two positive elliptic points and one negative elliptic point.
See Figure~\ref{fig:of_exam_2fol} for the entire picture of $\F(D)$. 
\begin{figure}[htbp]
\begin{center}
\strut\AffixLabels{\includegraphics*[scale=0.5, width=50mm]{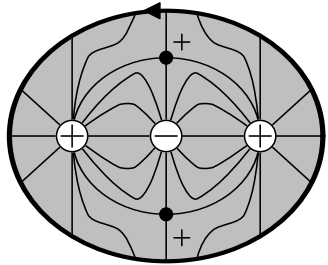}}
\caption{Example~\ref{informative example}. The open book foliation $\F(D)$.}
\label{fig:of_exam_2fol}
\end{center}  
\end{figure}
\end{example}

\subsection{Open book foliation vs. characteristic foliation}\label{sec:fromOB} 

{\quad}\\
Let $\cF(F)$ denote the characteristic foliation of a surface $F$ embedded in $(M, \xi)$. In this section, we compare the open book foliation $\F(F)$ and the characteristic foliation $\cF(F)$. 

\begin{theorem}[Structural stability]\label{identity theorem}
Assume that a surface $F$ in $M_{(S, \phi)}$ admits an open book foliation $\F(F)$. 
There exists a contact structure $\xi$ on $M_{(S, \phi)}$ supported by the open book $(S, \phi)$ such that 
$e_{\pm}(\F(F))=e_{\pm}(\cF(F))$ and $h_{\pm}(\F(F))=h_{\pm}(\cF(F))$.

Moreover, if $\F(F)$ contains no $c$-circles, then $\F(F)$ and $\cF(F)$ are topologically conjugate, namely there exists a homeomorphism of $F$ that takes $\F(F)$ to $\cF(F)$. 
In particular \cite[Lemma 2.1]{ev} implies that $F$ is a convex surface. 
\end{theorem}

\begin{proof}

Recall the Thurston-Winkelenkemper construction \cite{TW}, \cite[p.151-153]{geiges} of a contact structure compatible with the open book $(S, \phi)$: 

Away from the binding Thurston-Winkelenkemper's contact $1$-form is written as $\alpha = \beta_t + C dt$ where $t \in [0,1]$ (page parameter), $C\gg 1$ is a sufficiently large constant number and $\{\beta_t\}$ is a smooth family of $1$-forms on the page $S_t$ such that $d\beta_t$ is an area form of $S_t$ of total area $2\pi$ and $\beta_1=\phi^* \beta_0$. 
Such family $\{\beta_t\}$ is not unique, so we choose any to start with.

Near a binding component there exists cylindrical coordinates $(\theta, r, t)$, where $\theta$ represents the positive direction of the binding and $t \in [0, 1]$ is the same $t$ as above, such that 
\begin{equation}\label{eq:near-binding}
\alpha= 2 d\theta + r^2 dt.
\end{equation}

Assume that $F$ admits an open book foliation $\F(F)$. 
In the following we use the $1$-form on $M_{(S, \phi)}$, $\alpha$,  chosen above and contact planes $\xi:=\ker \alpha$ to study neighborhoods of singular and non-singular points. 

{\bf (Elliptic points)} 
Suppose $p \in \Int(F)$ is an elliptic point of $\sgn(p)=:\epsilon \in\{\pm1\}$. 
This means that a binding component, $\gamma$, transversely intersects $F$  at $p$ with sign $\epsilon$, see Fig~\ref{push-down}-(1). 
Take a disc neighborhood $D \subset F$ of $p$ whose open book foliation $\F(D)$ contains no other singularities. 
By (\ref{eq:near-binding}) we know that along $\gamma$ the contact planes and $\gamma$ transversely intersect with sign $+1$. 
We push down (or up) a very small neighborhood $D_0 \subset D$ of $p$ along $\gamma$ without touching the rest of the surface, see Fig~\ref{push-down}-(2). 
\begin{figure}[htbp]
\begin{center}
\SetLabels  
(.05*.9) (1)\\
(.5*.9) (2)\\
(.17*.55) $p$\\
(.2*.9) $\gamma$\\
(.9*.5) $\cF(D_0)$\\
(.1*.2) $\F(D)$\\
(.67*.16) $p$\\
(.9*.2) $p$\\
\endSetLabels
\strut\AffixLabels{\includegraphics*[height=30mm]{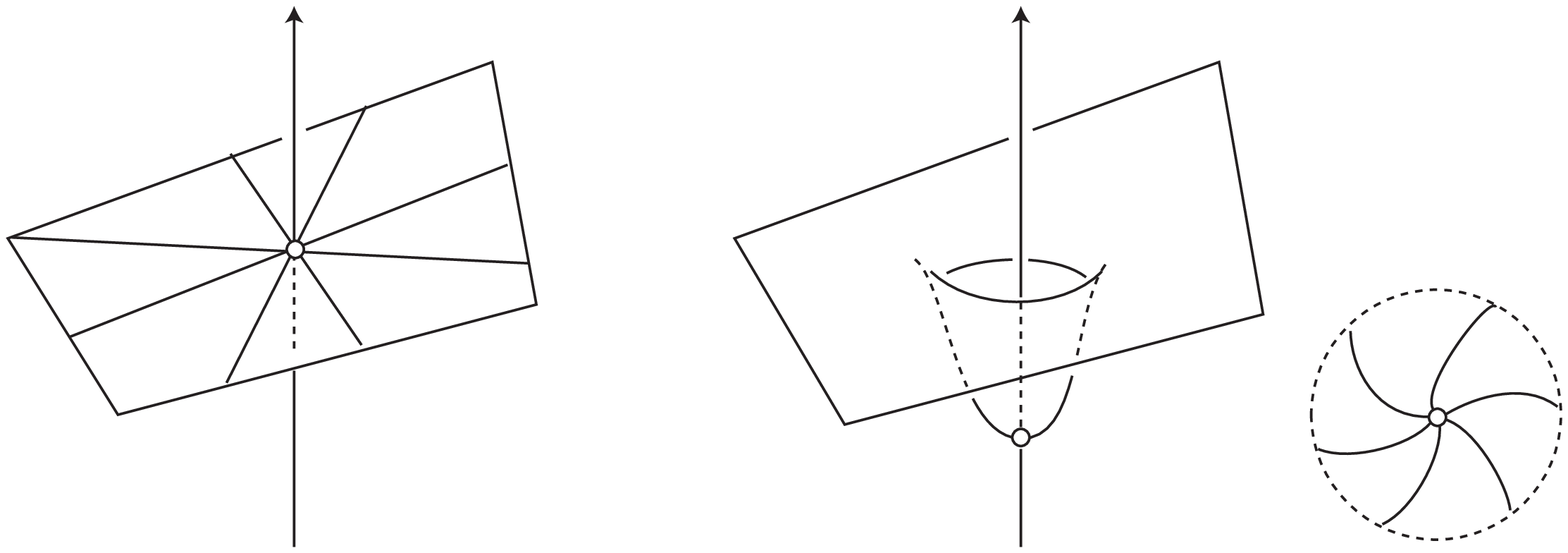}}
\caption{}\label{push-down}
\end{center}
\end{figure}
Since this operation preserves the open book foliation we may call the perturbed surface by the same name $F$. 
By the symmetry with respect to $\gamma$ of the pushed $D_0$ and $\alpha=2d\theta+ r^2 dt$, at the new $p=D_0 \cap \gamma$ the tangent plane and the contact plane satisfy $T_pF = \epsilon \cdot \xi_p$ hence the new $p$ is an elliptic point of the characteristic foliation $\cF(D_0)$ of sign $\epsilon$. 
(If we push {\em up} in Fig~\ref{push-down}-(2) instead of push {\em down} exactly the same argument holds.)

{\bf (Hyperbolic points)}
Let $p \in \F(F)$ be a hyperbolic point of $\sgn(p)=+1$. 
(A parallel argument holds for the negative case.) 
Take an open ball neighborhood $U \subset M_{(S, \phi)}$ of $p$ in which $p$ is the only singularity of the open book foliation $\F(F)$.  
Let $(x, y, z)$ be coordinates of $U$ such that 
\begin{itemize}
\item[(i)] $z$ is a coordinate for a sub-interval of $[0,1]$ such that $\partial_z= \partial_t$,   
\item[(ii)] $(x, y)$ are coordinates for the open disc $U\cap S_t$, 
\item[(iii)] $p=(0,0,0)$. 
\end{itemize}
We may assume that $F \cap U$ is a saddle surface and satisfies $z=x^2 - y^2.$
The normal vector $\vec{n}_F$ to the surface at $q=(x, y, z) \in F \cap U$ is $\vec{n}_F = (-2x, 2y, 1)$. 
Suppose that the contact plane $\xi_q = \ker(\alpha_q)$ at $q$ is spanned by 
\begin{equation}\label{spanning}
\xi_q= {\rm span} \langle \partial_x + f(q) \partial_z, \ \partial_y+g(q)\partial_z\rangle_\RR
\end{equation}
for some smooth functions $f, g: U\to \RR$. 
Let $\vec{n}_\xi :=(-f(q), -g(q), 1)$ then $\vec{n}_\xi$ is a positive normal to $\xi_q$. 
We have:  
$$
0=\alpha_q ( \partial_x + f(q) \partial_z) = \beta_q(\partial_x) + C f(q)
$$
$$
0=\alpha_q (\partial_y+g(q)\partial_z) = \beta_q(\partial_y) + C g(q)
$$ 
Since $C$ can be taken as large as we want we have: 
\begin{equation}\label{small}
|f(q)| = |\beta_q(\partial_x)|/C \ll 1,  \mbox{ and } 
|g(q)| = |\beta_q(\partial_y)|/C \ll 1
\end{equation}
Therefore, if we take $U$ small enough there exists a unique point $p_0 \in U\cap F$ at which $\vec{n}_F=\vec{n}_\xi$, and the foliations $\F(F\cap U)$ and $\cF(F \cap U)$ are topologically conjugate. 
In particular, $p_0$ is a hyperbolic point of the characteristic foliation and $\sgn(p_0)=\sgn(p)$.

{\bf (Non singular points)}
Let $p \in \Int(F)$ be a non-singular point in $\F(F)$. 
Take a small open $3$-ball neighborhood $U \subset M_{(S, \phi)}$ of $p$ so that the surface $F\cap U$ contains no singularity of $\F(F)$. 
Let $(x, y, z) \in \RR^3$ be coordinates of $U$ with the above conditions (i, ii, iii). 
We may suppose that $z=ky$ is satisfied on $F\cap U$ for some $k\neq 0$. 
So the leaves of $\F(U\cap F)$ are the integral curves of the vector field $\partial_x$.
Given a point $q=(x, y, z) \in U$ we may assume the above (\ref{spanning}) and (\ref{small}).
Hence the normal vector $\vec{n}_\xi=(-f(q), -g(q), 1)$ to $\xi_q$ and the normal vector $\vec{n}_F = (0, -k, 1)$ to $T_q F$ are not parallel to each other, i.e., $\xi_q \neq \pm T_q F$ and the point $q$ is not a singularity of $\cF(F)$.

The above arguments conclude the first assertion of the theorem:
$e_{\pm}(\F(F))=e_{\pm}(\cF(F))$ and $h_{\pm}(\F(F))=h_{\pm}(\cF(F))$. \\

To prove the second assertion, we assume that $\F(F)$ contains no $c$-circles. 
By Proposition~\ref{prop:region}, $F$ decomposes into type $aa$-, $ab$- and $bb$-tiles. 
For the stable separatrices $\mathcal S$ in each tile we take a small disc neighborhood $\mathcal D \subset F$ of $\mathcal S$.
The leaves of $\F(\mathcal D)$ are oriented outward along the boundary $\partial \mathcal D$. 
This implies that $\partial \mathcal D$ is a positive  braid w.r.t the open book $(S, \phi)$, or equivalently a positive transverse unknot in $(M_{(S, \phi)}, \xi)$, where $\xi$ is the contact structure chosen above.  
Therefore, the leaves of $\cF(\mathcal D)$ are also outward along $\partial \mathcal D$. 
Moreover, the above argument shows that $\cF(\mathcal D)$ and $\F(\mathcal D)$ are topologically conjugate relative to $\partial \mathcal D$. 

A similar argument holds for each unstable separatrices. 
Hence we conclude that $\cF(F)$ and $\F(F)$ are topologically conjugate. 
\end{proof}

\begin{remark}
\label{remark:rigid}
The above proof of Theorem~\ref{identity theorem} shows that the open book and characteristic foliations may coincide, especially when there are no $c$-circles in $\F$.  
Interesting contrast is found between open book foliations and  characteristic foliations (on convex surfaces). 
\begin{itemize}
\item
For a given closed surface $F$, we can always find a convex surface $F_{cv}$ that is isotopic and $C^{\infty}$-close to $F$. 
However, in general, there may not exist a surface admitting an open book foliation that is even $C^1$-close to $F$ (eg. when $F$ has local extrema relative to the pages and then we apply finger moves). 

\item 
The {\em dividing set} $\Gamma_F$ of a convex surface $F$ encodes essential information of local contact structure near $F$. 
It yields a decomposition $F\setminus\Gamma = F_{+} \sqcup F_{-}$ of $F$. 
If $\F(F)$ has no $c$-circles then the region $F_-$ is homotopy equivalent to our graph $G_{--}$.

\item 
In a characteristic foliation on a convex surface, any closed leaf is either repelling or attracting, and there are no type $ac$-, $bc$- and $cc$-hyperbolic points (Figure~\ref{region}) due to the Morse-Smale condition (cf. \cite[p.171]{geiges}).
On the other hand, an annular neighborhood of a $c$-circle in an open book foliation is foliated by parallel $c$-circles.

\item 
In the theory of convex surfaces, {\em Giroux elimination} \cite{giroux2}, \cite[Lemma 4.6.26]{geiges} allows us to remove a pair of elliptic and hyperbolic singularities of the same sign by an arbitrary $C^{0}$-small isotopy. 
Morally, one thinks that Giroux elimination corresponds to elimination of a certain arrangement of a pair of local extremum and a saddle point in an open book foliation $\F(F)$ by `flattening' the surface $F$. 
See Figure~\ref{fig:elimination}.

In a subsequent paper \cite{ik2} we discuss a number of operations in open book foliation theory that allow us to remove singularities. 
\end{itemize}
\begin{figure}[htbp]
\begin{center}
\SetLabels
(0*.65) elimination of\\
(0*.62) a local maximum\\
(0*.59) and a saddle\\
(.95*.29) Giroux elimination\\
(.85*.67) isotopy\\
(.96*.86) $+$\\
(.79*.87) $+$\\
(.96*.5) $+$\\
(.79*.51) $+$\\
\endSetLabels
\strut\AffixLabels{\includegraphics*[width=110mm]{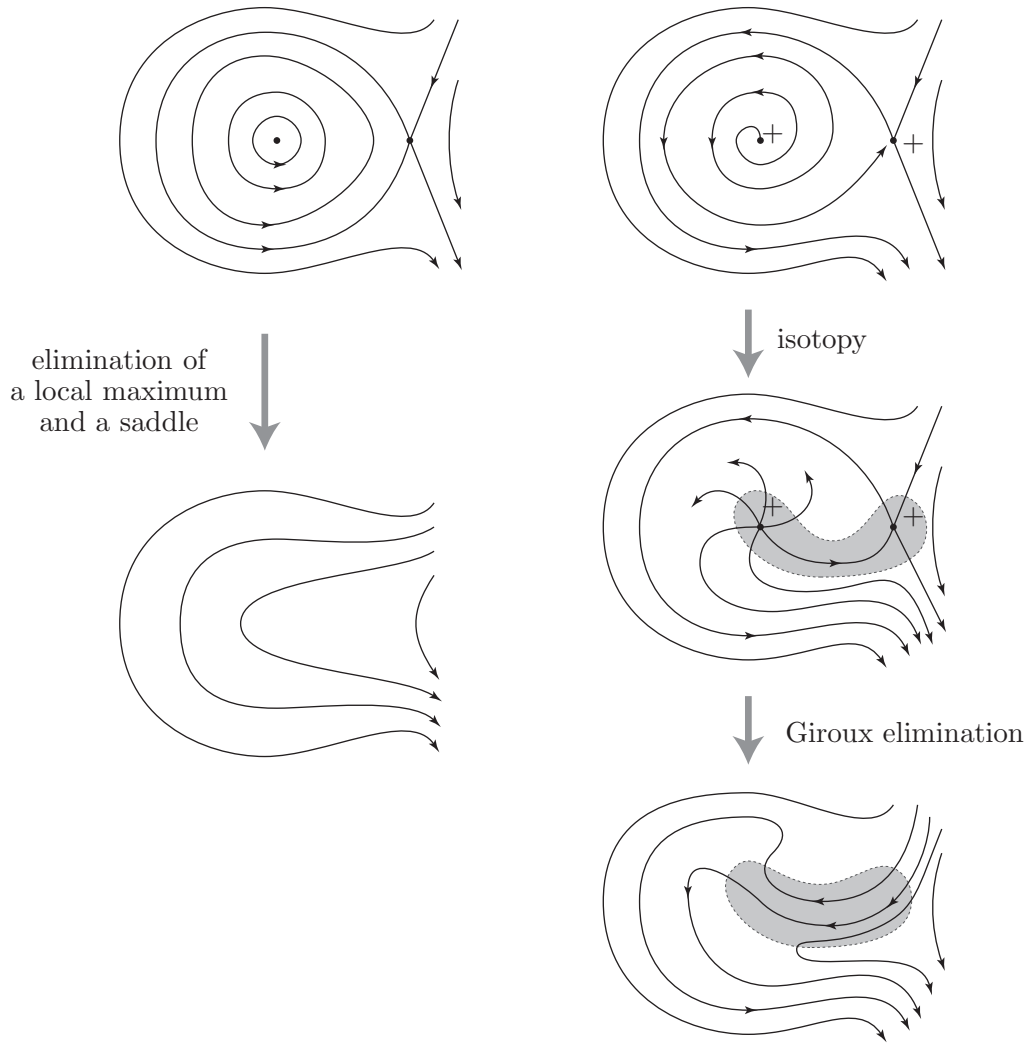}}
\caption{(Left) Elimination of a local extremum and a saddle point in open book foliation. (Right) Giroux elimination takes place in the shaded region of  characteristic foliation.}\label{fig:elimination}
\end{center}
\end{figure}
\end{remark}

\section{The self linking number}\label{sec:SL}

A {\em transverse knot} in a contact 3-manifold $(M, \xi)$ is an embedding of $S^1$ transverse to $\xi$. 
It is known that a transverse knot is a contact submanifold of a contact 3-manifold (see \cite[Rem 2.1.15]{geiges} for example). 
In this section we study an invariant of transverse knots, called the self linking number.

\begin{definition}
Let $L \subset (M, \xi)$ be a transverse link that bounds a surface $F$, i.e., $L$ is 0-homologous.
The rank 2 vector bundle $\xi|_F \to F$ over $F$ is trivializable. Let $s$ be a nowhere vanishing smooth section of the bundle. 
Push $L$ into the direction of $s$ and call the resulting link $L^{+s}$. 
The {\em self linking number} of $L$ relative to $[F] \in H_2(M, L; \Z)$, which we denote by $sl(L, [F])$, is the algebraic intersection number of $L^{+s}$ and $F$.   
\end{definition}

Using Mitsumatsu and Mori's theorem \cite{MM} or Pavelescu's  \cite{Pav, P2}, we can identify a transverse link in $(M, \xi)$ with a closed braid in any compatible open book $(S, \phi)$. 
The goal of this section is to prove Theorem~\ref{theorem:sl-formula}, a self-linking number formula for closed braids.

Our strategy is to construct a special Seifert surface $\Sigma$ for a given closed braid and count the singularities of its open book foliation $\F(\Sigma)$ then apply the following proposition:

\begin{proposition}\label{sl-formula-1}
Suppose that $F \subset M_{(S, \phi)}$ is a surface with the open book foliation $\F(F)$. In particular, $\partial F$ is a transverse link in $(M_{(S, \phi)}, \xi_{(S, \phi)})$.
Recall the integers $e_\pm = e_\pm(\F(F))$, $h_\pm=h_\pm(\F(F))$  defined in Definition~\ref{def of e}. 
We have
$$sl(\partial F, [F]) = -\langle e(\xi),[F] \rangle = -(e_{+}-e_{-})+(h_{+}-h_{-}).$$
\end{proposition}

\begin{proof}
The self-linking number formula in characteristic foliation theory, see \cite[p.203]{geiges} for example, together with Theorem~\ref{identity theorem} yields the above formula.
\end{proof}

In order to state our main theorem (Theorem~\ref{theorem:sl-formula}) we first need to define a function $c: \MCG(S) \times H_{1}(S;\partial S) \rightarrow \Z$ in \S\ref{sec:def_of_c}. 
Later in \S\ref{sec:surface-one-boundary} we show that the function $c$ is  related to the first Johnson-Morita homomorphism, a well-studied homomorphism in mapping class group theory.

\subsection{Definition of function $c$.}\label{sec:def_of_c}

{\quad }\\
Let $S=S_{g, r}$ be an oriented genus $g$ surface with $r$ boundary components. 
We divide the surface $S$ by {\em walls} (dashed arcs in Figure~\ref{braid-generator}) into $g+r-1$ {\em chambers} so that $g$ of which are once-punctured tori and $r-1$ of which are annuli.

\begin{definition}[Normal form]\label{def of normal form}
A relative homology class $a \in H_{1}(S,\partial S)$ is represented by a set of properly embedded oriented simple closed curves and arcs in $S$. 
Among such multi-curve representatives, we take a special one, $N(a)$, which satisfies the following conditions: 
\begin{itemize}
\item 
$N(a)$ does not intersect the walls.
\item
Any subset of $N(a)$ has non-trivial homology in $H(S,\partial S)$,
i.e., the components of $N(a)$ in a torus (resp. an annulus) chamber is a torus knot or link (resp. parallel arcs joining $\gamma_0$ and $\gamma_i$ in Figure~\ref{braid-generator}) oriented in the same direction.
\end{itemize}
Clearly the multi-curve $N(a)$ is uniquely determined up to isotopy. We call $N(a)$ the {\em normal form} of the homology class $a \in H_1(S, \partial S)$.
\end{definition}

\begin{definition}[OB cobordism]\label{def of OB cobordism}
Let $A$ and $A'$ be oriented, properly embedded multi-curves in $S$ representing the same homology class $[A]=[A'] \in H_{1}(S,\partial S)$. 
An {\em open book foliation cobordism} ({\em OB cobordism}) between $A$ and $A'$, denoted by $A \stackrel{\Sigma}{\rightarrow} A'$, is a properly embedded oriented compact surface $\Sigma$ in $S \times [0,1]$ such that: 
\begin{itemize}
\item 
$\Sigma \cap S_0 = \partial\Sigma \cap S_0 = -A \times \{0\}$.
\item 
$\Sigma \cap S_1 =  \partial\Sigma \cap S_1 =A' \times \{1\}$.
\item 
$\partial A= A \cap \partial S = A' \cap \partial S= \partial A'$. 
\item 
$\partial \Sigma = (-A \times \{0\}) \cup (A' \times \{1\}) \cup (\partial A \times [0,1])$.
\item The fibration $\{S_t\}_{t \in [0,1]}$ induces a foliation $\mF_\Sigma$ on $\Sigma$ all of which singularities are of hyperbolic type.
\end{itemize}
\end{definition}

\begin{proposition}\label{prop existence of OB corbordism} 
There is an OB cobordism $A  \stackrel{\Sigma}{\rightarrow} N(a)$ for any multi-curve representative $A$ of $a \in H_1(S, \partial S)$. 
That is, if multi-curves $A$ and $A'$ represent the same homology class then there exists an OB cobordism $A\stackrel{\Sigma}{\rightarrow} A'$.
\end{proposition}

\begin{proof}
We construct an oriented surface $\Sigma$ embedded in $S\times [0,1]$ with $\Sigma \cap S_0=-A$ and $\Sigma \cap S_1=N(a)$.
Let $w$ be one of the walls. 
Since $[A]=[N(a)]\in H_1(S, \partial S)$ and the normal form $N(a)$ does not intersect $w$, the algebraic intersection number $[A] \cdot w=0$. 
We take a collar neighborhood $\nu(w)\subset S$ of $w$ so that each component of $\nu(w) \cap A$ has geometric intersection number $1$ with $w$. 
The arcs $\nu(w) \cap A$ may not all have the same orientation. 
As $t\in [0, 1]$ increases we apply the configuration changes to pairs of consecutive arcs in $\nu(w) \cap A$ with opposite orientations as in the passage of Figure~\ref{configure-change}
\begin{figure}[htbp]
\begin{center}
\SetLabels
(.05*1) $w$\\
(.22*.6) $-$\\
(.77*.6) $+$\\
\endSetLabels
\strut\AffixLabels{\includegraphics*[height=1.4cm]{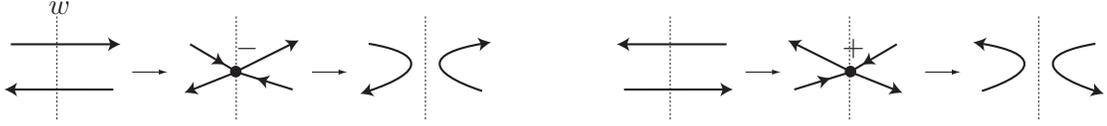}}
\caption{Configuration change of $\nu(w)\cap A$.}\label{configure-change}
\end{center}
\end{figure}
until we remove all the arcs of $\nu(w) \cap A$.
Each configuration change introduces a new hyperbolic singularity. 
We repeat the procedure for all the walls. The deformed multi-curve $A$, which we denote $A'$, no longer intersects the walls.

The multi-curve $A' \subset S$ may contain null-homologous sets of $c$-circles. 
We remove them by the following three steps.

{\bf (Step 1)}
If there exist $c$-circles bounding concentric discs in a chamber $H$ of $S$ and oriented in the same direction, then we remove them from the outermost one. 
We can find a describing arc of a hyperbolic point (cf. Figure~\ref{fig:hyperbolic}) that joins the outermost $c$-circle and some curve in  $A' \cap H$ and is properly embedded in $H \setminus A'$. 
As shown in the top row of Figure \ref{step12} one hyperbolic singularity is introduced then the $c$-circle disappears. 
The sign of the hyperbolic singularity is $+1$ if and only if the $c$-circle is oriented clockwise. 
\begin{figure}[htbp]
 \begin{center}
   \includegraphics[width=80mm]{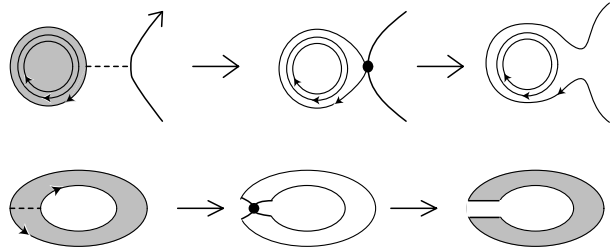}
   \caption{Step 1 (top) and Step 2 (bottom).}
\label{step12}
 \end{center}
\end{figure}

{\bf (Step 2)}
If there is a pair of $c$-circles with opposite orientations that bounds an annulus in $S \setminus A'$, then remove the pair by introducing a hyperbolic singularity of sign $\e$ between the two $c$-circles as in Figure \ref{step12}. 
The resulting $c$-circle bounds a disc that can be removed by Step 1 with the expense of another hyperbolic singularity of sign $-\e$.

{\bf (Step 3)}
Let $H$ be a once-punctured torus chamber of $S$.
After Steps 1, 2, there exist $p, q, r \in \Z$ such that in $H$ the multi-curve $A'$ is the union of $(p, q)$ torus link and $r$ boundary parallel $c$-circles oriented in the same direction.
As in Figure~\ref{Step-4} we remove the $c$-circles by introducing $r$ hyperbolic points of the same sign 
The sign depends on the signs of $p, q$ and  the orientation of the boundary parallel $c$-circles. 
\begin{figure}[htbp]
\begin{center}
\SetLabels
(.02*.9) $H$\\
(.02*.3) $H$\\
(.27*.8) $p$\\
(.27*.25) $p$\\
(.14*.94) $q$\\
(.14*.37) $q$\\
(0*.67) $r$\\
(0*.1) $r$\\
(.63*.8) $p$\\
(.47*.94) $q-1$\\
(.43*.63) $r-1$\\
(1.05*.65) $r-1$\\
(1.05*.1) $r-1$\\
(1*.25) $p$\\
(.85*.38) $q$\\
\endSetLabels
\strut\AffixLabels{\includegraphics*[height=6cm]{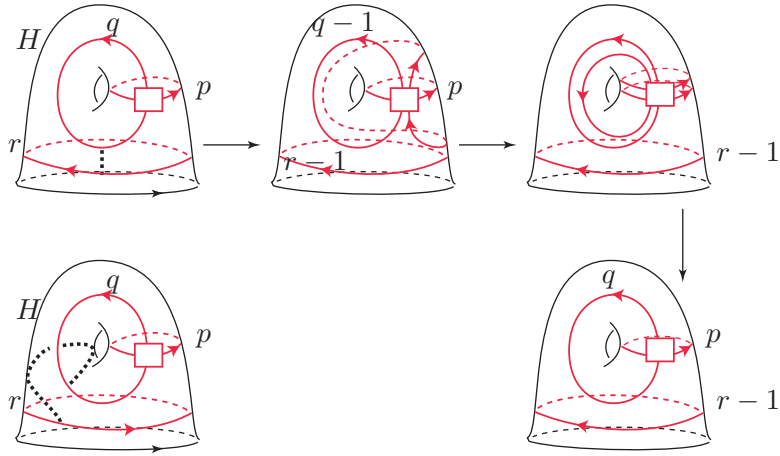}}
\caption{Step 3. Remove boundary parallel null-homologous $c$-circles by configuration changes along the dashed arcs.}\label{Step-4}
\end{center}
\end{figure}

Now $A'$ is deformed to the normal form $N(a)$. Hence we get a desired surface $\Sigma$. 
\end{proof}

\begin{proposition}
For an OB cobordism $A \stackrel{\Sigma}{\rightarrow} A'$ let $h_{+}(\mF_{\Sigma})$ (resp. $h_{-}(\mF_{\Sigma})$) denote the number of the positive (resp. negative) hyperbolic singularities of $\mF_{\Sigma}$. 
The value $$d(A \stackrel{\Sigma}{\rightarrow} A'):=h_{+}(\mF_\Sigma) -h_{-}(\mF_\Sigma)$$ is independent of the choice of cobordism surface $\Sigma$ and it only depends on the multi-curve representatives $A$ and $A'$.  
Hence we may denote 
$$d(A,A'):=d(A \stackrel{\Sigma}{\rightarrow} A').$$ 
\end{proposition}

\begin{proof}
Suppose that $A \stackrel{\Sigma'}{\rightarrow} A'$ is another OB cobordism. 
We embed $-\Sigma'$ in $S\times[1, 2]$.
We glue $\Sigma$ and $-\Sigma'$ at the page $S_{1}$ and obtain a surface $F'$ in $S \times [0,2]$.
Since $F' \cap S_2 = A = -(F' \cap S_0)$, we can further identify $F' \cap S_2$ and $F' \cap S_{0}$ by the identity map that defines a surface $F$ embedded in the open book $(S, \rm{id})$. 
Since a $\pm$-hyperbolic singular point in $\mF_{\Sigma'}$ turns to a $\mp$-hyperbolic point in $\mF_{-\Sigma'}$ we have $d(A' \stackrel{-\Sigma'}{\rightarrow} A) = -d(A \stackrel{\Sigma'}{\rightarrow} A')$ and
\begin{equation}\label{hyp pts of F_*}
h_+(\F(F))- h_-(\F(F)) = d(A \stackrel{F}{\rightarrow} A)=
d(A \stackrel{\Sigma}{\rightarrow} A') - 
d(A \stackrel{\Sigma'}{\rightarrow} A').
\end{equation}
By Definition~\ref{def of OB cobordism} the elliptic points in $\F(F)$ correspond to the lines $\partial A \times [0,2]$. 
Since the endpoints of each arc component of $A$ correspond to two elliptic points of opposite signs we get 
\begin{equation}\label{elliptic points of F_*}
e_{+}(\F(F)) = e_{-}(\F(F)).
\end{equation}
Let $\xi_{id}$ be the contact structure supported by the open book $(S,id)$.
Since the Euler class of $\xi_{id}$ is equal to zero, by Proposition \ref{sl-formula-1}, (\ref{hyp pts of F_*}) and (\ref{elliptic points of F_*}), we have
$$0 = \langle e(\xi_{id}), [F] \rangle = d(A \stackrel{\Sigma}{\rightarrow} A') - d(A \stackrel{\Sigma'}{\rightarrow} A').$$
\end{proof}

We are ready to define the function $c([\phi], a)$. 
The following definition is geometric. 
Later we study algebraic properties of $c([\phi], a)$ in Propositions~\ref{c-planar}, \ref{c-genus} and Theorem~\ref{theorem:c-formula}.

Let $\MCG(S)$ denote the {\em mapping class group} of $S$, that is the group of isotopy classes of orientation preserving homeomorphisms of $S$ fixing the boundary $\partial S$ pointwise.

\begin{definition}\label{def of function c}
Let $[\phi] \in \MCG(S)$ and $a \in H_{1}(S,\partial S)$. 
Define
\[ c([\phi], a) := d(\phi ( N(a)), N(\phi_{*}a)). \]

In general, the multi-curve $\phi(N(a))$ may not be isotopic to $N(\phi_{*}(a))$.
But if $\phi(N(a))$ is isotopic to $N(\phi_{*}(a))$ we can choose an OB cobordism $\phi(N(a))\stackrel{\Sigma}{\rightarrow}N(\phi_{*}(a))$ to be product $\Sigma\simeq\phi(N(a))\times [0,1]$ with no hyperbolic singularities, hence $c([\phi],a) = 0$. We call such an OB cobordism {\em trivial}.
\end{definition}

\subsection{A self-linking number formula for braids}\label{sec:sl-formula} 

{\quad}\\
Let $S=S_{g, r}$ be an oriented genus $g$ surface with $r$ boundary components $\gamma_1, \ldots, \gamma_r$. 
The orientation of $\gamma_i$ is induced from that of $S$.
Let $b$ be an $n$-stranded braid in $S \times [0,1]$ with $b \cap S_1= b \cap S_0 = \{x_1, \ldots, x_n \} \subset S$.
By braid isotopy we may assume that points $x_1, \ldots, x_n$ are lined up in this order on an arc parallel to and very close to $\gamma_0$. 
The arc $\{x_i\}\times[0,1]$ is called the {\em$i$-th braid strand} in $S\times [0,1]$.
We define oriented loops, $\rho_i \subset S$ ($i=1, \ldots, 2g+r-1$) with the base point $x_n$ as in Figure~\ref{braid-generator}. 
\begin{figure}[htbp] 
\begin{center}
\SetLabels
(0*.28) $\upsilon_1$\\
(.08*.2) $\gamma_0$\\
(.83*.53) $\gamma_1$\\
(.57*.53) $\gamma_{r-1}$\\
(.81*.38) $\rho_1$\\
(.54*.4) $\rho_{r-1}$\\
(.47*.53) $\rho_{r}$\\
(.4*.8)  $\rho_{r+1}$\\
(.3*.4) $\rho_{2g+r-2}$\\
(.17*.85) $\rho_{2g+r-1}$\\
(.86*.78) $\rho_1'$\\
(.6*.9) $\rho_{r-1}'$\\
(.13*.15) $y_1$\\
(.48*.03) $y_n$\\
(.34*0) $\upsilon_n$\\
(.16*.35) $x_1$\\
(.51*.16) $x_n$\\
(.8*.73) wall\\
\endSetLabels
\strut\AffixLabels{\includegraphics*[height=5cm]{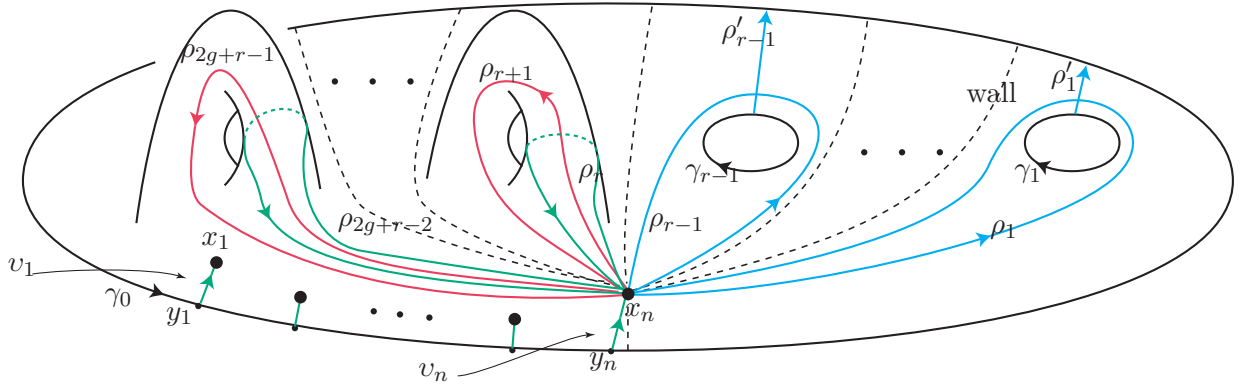}}
\caption{Surface $S$.}\label{braid-generator}
\end{center}
\end{figure}
Geometrically $\rho_i$ represents the $n$-th braid strand winding along $\rho_i$ as $t\in [0, 1]$ increases. 
Let $\sigma_i$ denote the positive half twist of the $i$-th and the $(i+1$)-th braid strands. 
As a consequence of the Birman exact sequence \cite{b}, the braid $b$ is represented by a braid word 
$b_1^{\e_1} b_2^{\e_2} \cdots b_l^{\e_l}$ (read from the left) where
$b_i\in\{\rho_1, \ldots, \rho_{2g+r-1}, \sigma_1, \ldots, \sigma_{n-1}\}$ and $\e_i\in\Z \setminus\{0\}$.

Fix a diffeomorphism $\phi \in {\rm Aut}(S, \partial S)$.
Since $x_i$ is near $\gamma_1 \subset \partial S$, we have $\phi(x_i)=x_i$ and identify $\{x_i\}\times \{1\}$ and $\{x_i\}\times \{0\}$ under $\phi$ that yields a closed braid $\hat b$ in $M_{(S, \phi)}$. 
We assume that $\hat b$ is null-homologous in the rest of the section.

\begin{claim}\label{claim:a}
Put $[b]=\sum_{i=1}^l \e_i [b_i] \in H_1(S; \Z)$, where we set $[\sigma_k]=0$ for $k=1, \ldots, n-1$.
Then there exists a {\em (not necessarily unique)}  homology class $a \in H_1(S, \partial S;\Z)$ such that $[b] = a - \phi_*(a)$ in $H_1(S;\Z)$.
\end{claim}

\begin{proof}
The homology group of the manifold $M_{(S, \phi)}$ is  computed by Etnyre and Ozbagci \cite[p.3136]{eo}:
$$
H_1(M_{(S, \phi)}; \Z)= 
\left\langle
[\rho_1], \ldots, [\rho_{2g+r-1}] \ |\
[\rho_i'] - \phi_* [\rho_i'] = 0, \quad 
i=1, \ldots, 2g+ r-1 \\
\right\rangle,
$$
where 
$$
\rho_i' = \left\{
\begin{array}{l}
\mbox{a properly embedded arc from } \gamma_0 \mbox{ to } \gamma_i \mbox{ and dual to }\rho_i, \mbox{ for }i=1, \ldots, r-1,\\
\rho_i,  \mbox{ for }i=r, \ldots, 2g+ r-1.
\end{array}\right.
$$
Though $\rho_i'$ is an arc for $i=1, \ldots, r-1$, since $\phi= {\rm id}$ on $\partial S$, we can view $\rho_i' \cup \phi(-\rho_i')$ as an oriented (immersed) loop in ${\rm Int}(S)$. Then we consider $[\rho_i'] - \phi_* [\rho_i']\in H_1(S; \Z)$ representing the loop $\rho_i' \cup \phi(-\rho_i')$.

Since $[\hat b] = 0$ in $H_1(M;\Z)$, there exist $s_i \in \Z$ for $i=1, \ldots, 2g+ r-1,$ such that 
$$[b]= \sum_{i=1}^{2g+ r-1} s_i ([\rho_i'] - \phi_* [\rho_{i}']) \ \mbox{ in } \ H_1(S; \Z).$$
Hence if we put $a=\sum_{i=1}^{2g+ r-1} s_i [\rho_i']$, under the identification $[\rho_i'] - \phi_* [\rho_i']= [\rho_i' \cup \phi(-\rho_i')]$, we have $[b]=a - \phi_*(a)$. 
\end{proof}

\begin{definition}
For homology classes $[a_1] \in H_1(S, \partial S; \Z)$ and $[a_2] \in H_1(S; \Z)$ we denote the {\em algebraic intersection number} by $[a_1] \cdot [a_2] \in \Z$. It counts the transverse intersections of representatives $a_1$ and $a_2$ algebraically in the way described in Figure~\ref{crossing-count}.
For example, we have $[\rho_1']\cdot [\rho_1]= 1$ and $[\rho_{r}] \cdot [\rho_{r+1}] = 1$.
\begin{figure}[htbp]
\begin{center}
\SetLabels
(.3*.9) $a_1$\\
(.7*.9) $a_1$\\
(-.03*.9) $a_2$\\
(1.03*.9) $a_2$\\
(.17*.45) $+$\\
(.9*.45) $-$\\
\endSetLabels
\strut\AffixLabels{\includegraphics*[height=2cm]{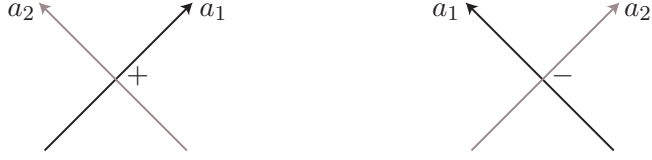}}
\caption{Algebraic intersection number $[a_1] \cdot [a_2]$.}\label{crossing-count}
\end{center}
\end{figure}
\end{definition}

Here is our main theorem of this section:

\begin{theorem}[Self linking number formula]\label{theorem:sl-formula}
Let $[b] \in H_1(S;Z)$ and $\hat b$ be as above. 
Let $a \in H_1(S, \partial S;\Z)$ be a homology class such that $[b] = a - \phi_*(a)$ in $H_{1}(S;\Z)$ $($see Claim~\ref{claim:a}$)$.
Recall the function $c([\phi], a)$ in Definition~\ref{def of function c}. 
For the choice of $a\in H_1(S, \partial S;\Z)$ there exists a Seifert surface $\Sigma=\Sigma_a$ of $\hat{b}$ such that the self-linking number satisfies the formula:
\begin{equation}\label{eq:self-linking}
sl(\hat{b}, [\Sigma]) = -n + \widehat{\exp}(b) - \phi_*(a)\cdot[b] + c([\phi] , a), 
\end{equation}
where 
$$\widehat{\exp}(b)= \sum_{i=1}^l \e_i-\sum_{1\leq j<i \leq l} \e_i \e_j [b_j] \cdot [b_i].$$
\end{theorem}

\begin{remark}
The formula (\ref{eq:self-linking}) is a generalization of Bennequin's self linking formula of braids in the open book $(D^2, \id)$ \cite{Ben} and it also covers the works in \cite{kp} and \cite{k}. 
When $(S, \phi)=(D^{2}, \id)$ the function $\widehat{\exp}$ is equal to the usual exponent sum, $\exp: B_{n} \rightarrow \Z$, for the Artin braid group $B_n$ and  $\phi_*(a) \cdot [b] = c([\phi] , a) = 0$. 
Thus the formula (\ref{eq:self-linking}) contains Bennequin's self linking formula
\[ sl (\hat{b}) = -n + \exp(b). \]
With more elaborate investigation of the function $c$ we will deduce the self-linking number formulae of \cite{kp}, \cite{k} in Corollary~\ref{cor:sl-planar} below. 
\end{remark}

\begin{proof}

For each $i=1, \ldots, n$, take a point $y_i$ on the binding component $\gamma_0$ near $x_i$ so that $y_1, \ldots, y_n$ lined up in this order with respect to the orientation of $\gamma_0$, see Figure~\ref{braid-generator}.
Choose a properly embedded arc $\upsilon_i$ from $x_i$ to $y_i$ that contained in a small collar neighborhood of $\gamma_0$ so that $\phi(\upsilon_i)=\upsilon_i$. We require that $\upsilon_1, \ldots, \upsilon_n$ are mutually disjoint.

\vspace{2mm}

(\textit{Construction of surface $\Sigma_* \subset S \times [0, \frac{1}{2}]$}) \ 
Fix $a\in H_1(S, \partial S;\Z)$ with $a-\phi_*(a)=[b]$. 
Let $N(a)$ denote the normal form of $a$, see Definition~\ref{def of normal form}. 
Let $A_1, A_{1/2}, A_0$ be oriented multi-curves in $S$ defined by: 
\begin{eqnarray*}
A_1 &=& \upsilon_1 \cup \cdots \cup \upsilon_n \cup N(a)\\
A_{1/2} &=& \upsilon_1\cup \cdots \cup \upsilon_n \cup N(\phi_* a)\\
A_0 = \phi(A_1) &=& \upsilon_1 \cup \cdots \cup \upsilon_n \cup \phi(N(a)) 
\end{eqnarray*}
Unlike $A_1$ or $A_{1/2}$ the multi-curve $A_0$ possibly intersects the walls. 
We have:   
\begin{eqnarray*}
[A_1] &=& a\\
{[A_{1/2}]}  &=& [A_0]  = \phi_* a
\end{eqnarray*}
Let $\phi(N(a)) \stackrel{\Sigma_\circ}{\rightarrow} N(\phi_*a)$ 
be an OB cobordism whose existence is guaranteed by  Proposition~\ref{prop existence of OB corbordism}. 
We compress $\Sigma_\circ$ vertically to fit in $S \times [0, 1/2]$ and take disjoint union with the vertical rectangle strips $(\upsilon_1 \cup \cdots \cup \upsilon_n) \times [0, 1/2]$. 
We call the resulting surface $\Sigma_*$.  
By the construction $$\Sigma_* \cap S_0 = - A_0, \quad \Sigma_* \cap S_{1/2} = A_{1/2},$$
where $S_0=S \times \{0\}$ and $S_{1/2}=S\times\{1/2\}$ are pages of the open book, and by Definition~\ref{def of function c} the algebraic count of the hyperbolic points of $\F(\Sigma_*)$ is $c([\phi], a)$.

\vspace{2mm}

(\textit{Construction of surface $\Sigma_{**} \subset S \times [\frac{1}{2}, 1]$}) \ 
The next goal is to construct an oriented surface $\Sigma_{**}$ embedded in $S\times [\frac{1}{2},1]$ with $\Sigma_{**} \cap S_1=A_1$ and $\Sigma_{**} \cap S_{1/2}=-A_{1/2}$.
Recall that $b$ is represented by the braid word $b_1^{\e_1} \cdots b_l^{\e_l}$. Let $I_i = [\frac{l+i-1}{2l}, \frac{l+i}{2l}]$ then $[\frac{1}{2}, 1]= I_1 \cup \cdots \cup I_l$.
We will build an oriented surface $\Sigma_i$ embedded in $S \times I_i$ inductively from $i=1$ to $l$ such that: 
\begin{enumerate}
\item
$\Sigma_1 \cap S_{1/2} = -A_{1/2}$ and $\Sigma_l \cap S_1 = A_1$.
\item
$\Sigma_{i} \cap S_{(l+i)/2l} = -(\Sigma_{i+1} \cap S_{(l+i)/2l})$. 
We denote this multi-curve on the page $S_{(l+i)/2l}$ by $A_{(l+i)/2l}$. 
\item
$A_{(l+i)/2l}$ does not intersect the walls. 
\item
$A_{(l+i)/2l}$ contains $\upsilon_1 \cup \cdots \cup \upsilon_n$ and 
any subset of $A_{(l+i)/2l} \setminus (\upsilon_1 \cup \cdots \cup \upsilon_n)$ has non-trivial homology in $H_1(S, \partial S)$,
\item 
$\partial \Sigma_i \cap (S \times \Int(I_i)) = b_i^{\e_i},$ so 
$[A_{(l+i)/2l}] = [A_{1/2}] + \e_1 [b_1] + \cdots + \e_i [b_i]$ in $H_1(S, \partial S)$.
\end{enumerate}
Eventually we will define $\Sigma_{**}=\Sigma_1 \cup \cdots \cup \Sigma_l$. 
Suppose that we have constructed $\Sigma_1, \ldots, \Sigma_{i-1}$ satisfying the above conditions.

{\bf(Case 1)} 
If the braid word $b_i=\sigma_k$, then as $t \in I_i$ increases, apply the deformation of the graph $A_{(l+i-1)/2l}$ as in the passage of Figure~\ref{sigma}
\begin{figure}[htbp]
\begin{center}
\SetLabels
(.03*.29) $y_k$\\
(.01*.45) $\upsilon_k$\\
(.11*.29) $y_{k+1}$\\
(.14*.45) $\upsilon_{k+1}$\\
(.89*.29) $y_k$\\
(.96*.29) $y_{k+1}$\\
(.87*.45) $\upsilon_k$\\
(1*.45) $\upsilon_{k+1}$\\
\endSetLabels
\strut\AffixLabels{\includegraphics*[height=3.5cm]{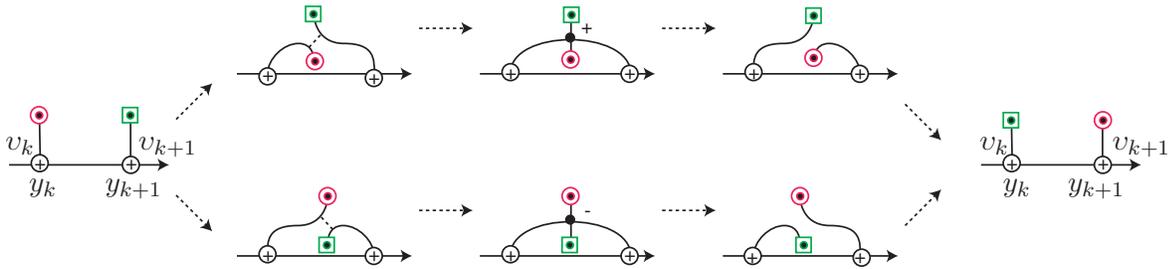}}
\caption{(Case 1) Deformation of graph $\upsilon_k \cup \upsilon_{k+1}$  corresponding to $\sigma_k$ (top) and $\sigma_k^{-1}$ (bottom). }\label{sigma}
\end{center}
\end{figure}
(where $\odot$ and $\boxdot$ denote the intersection of the braid $b$ and the page $S_t$) for $|\e_i|$ times that takes place in a small neighborhood of $\upsilon_k$ and $\upsilon_{k+1}$.
We call the surface that the graph traces out $\Sigma_i$.
The surface $\Sigma_i$ satisfies the above conditions (1)--(5) and the open book foliation $\F(\Sigma_i)$ has $|\e_i|$ hyperbolic singularities of ${\rm sgn}(\e_i)$.

{\bf(Case 2-1)} 
Suppose that $b_i = \rho_k$ and $\e_i=1$. 
Let $H$ be the chamber that $\rho_k$ belongs to.

Assume that $r \leq k \leq 2g+r-1$ so that $H$ is a torus with connected boundary.
For simplicity, put $u = A_{(l+i-1)/2l}$. 
By conditions (3), (4) above, we may assume that $u \cap H$ is some $(p,q)$-torus link. 
As $t \in I_i$ increases, move the point $x_n$ along $\rho_k$.  
See Figure~\ref{insertion-pair}. 
\begin{figure}[htbp]
\begin{center}
\SetLabels
(.15*.52) $x_n$\\
(.03*.84)  $H$\\
(.27*.8) $p$\\
(.9*.25) $p$\\
(.15*.95) $q$\\
(.75*.37) $q+1$\\
\endSetLabels
\strut\AffixLabels{\includegraphics*[height=7cm]{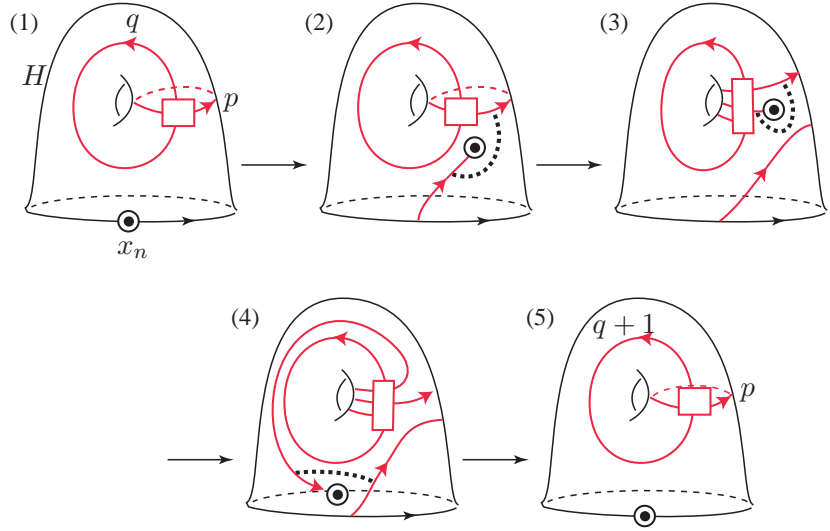}}
\caption{(Case 2-1) Construction of surface $\Sigma_i$ when $b_i^{\e_i}=\rho_k$.}\label{insertion-pair}
\end{center}
\end{figure}
To come back to the original position $x_n$ has to traverse $u$, which yields $p= [u] \cdot [\rho_k]$ many negative (=$-{\rm sgn}(\e_i)$) hyperbolic points. 
Moreover, the last step (Sketch (4)) adds one more hyperbolic singularity of positive ($={\rm sgn}(\e_i)$) sign. 
This defines the surface $\Sigma_i$ in $S \times I_i$. 
In summary, the value $h_+ - h_-$ increases by 
$${\rm sgn}(\e_i) \cdot 1 - [u] \cdot [\rho_k]$$ and the class $[u] \in H_1(S, \partial S)$ is replaced by $[u]+ [\rho_k]$ (compare Sketches (1) and (4)). 
No circle bounding a disc in $S$ has been created.

When $k=1, \ldots, r-1$ (i.e., the chamber $H$ is an annulus) a parallel  argument holds and the value $h_+ - h_-$ increases by ${\rm sgn}(\e_i) \cdot 1 - [u] \cdot [\rho_k]$.

{\bf(Case 2-2)} 
If $b_i=\rho_k$ and $\e_i \neq 1, 0$, repeat the above construction $|\e_i|$ times. 
Since $([u]+ [\rho_k])\cdot [\rho_k]  = [u] \cdot [\rho_k]$, the total change in $h_+ - h_-$ is 
\begin{eqnarray}\label{change in hyp points}
\e_i - \e_i [u]\cdot [\rho_k]
&\stackrel{(5)}{=}& 
\e_i - \e_i ([A_{1/2}] + \e_1 [b_1] + \cdots + \e_{i-1} [b_{i-1}]) \cdot [b_i] \\
&=& 
\e_i - (\sum_{j=1}^{i-1} \e_i \e_j [b_j] \cdot [b_i] ) - \e_i [A_{1/2}]\cdot [b_i]. \notag
\end{eqnarray}

After constructing $\Sigma_1, \ldots, \Sigma_l$, we glue them and obtain a desired surface $\Sigma_{**}$ in $S \times [1/2, 1]$ which increases the algebraic count of the hyperbolic singularities by
$$h_+ - h_- = \sum_{i=1}^l \e_i - (\sum_{i=1}^l \sum_{j=1}^{i-1} \e_i \e_j [b_j] \cdot [b_i] ) -  [A_{1/2}]\cdot[b].$$

Finally, we glue $\Sigma_*$ and $\Sigma_{**}$ by identifying $\Sigma_* \cap S_{1/2} = -(\Sigma_{**}\cap S_{1/2})$ and  $-(\Sigma_* \cap S_0) = \phi(\Sigma_{**} \cap S_1)$, which yields a Seifert surface $\Sigma$ for $\hat b$ in the open book $(S, \phi)$. 
By the construction, it is clear that $y_1, \ldots, y_n \in \gamma_0$ are positive elliptic points and the end points of arc $\rho_i'$ are elliptic points with distinct signs. 
By Proposition~\ref{sl-formula-1}, we obtain our self linking formula (\ref{eq:self-linking}).
\end{proof}

\subsection{Properties of the function $c$} 

\quad\\
In this section we study properties of the function $c$ in the self linking number formula (\ref{eq:self-linking}).
We will use the properties repeatedly in the later sections to deduce algebraic descriptions of the function $c$, which is originally defined geometrically.

\begin{proposition}\label{property_of_c}
Let $[\phi], [\psi] \in \MCG(S)$ be the mapping classes of $\phi, \psi \in Aut(S, \partial S)$ and $a,a' \in H_{1}(S, \partial S)$. We have:
\begin{enumerate}
\item $c([\phi], a + a') = c([\phi], a) + c([\phi], a')$. 
\item $c([\psi \phi], a) = c([\phi], a) + c([\psi],\phi_{*}(a))$.
\item Let $C$ be a simple closed curve which does not intersect the walls. Let $T_C$ denote the right-handed Dehn twist along $C$. We have $c([T_{C}],a) = 0$ for any $a$.
\item Let $C$ be a simple closed curve in $S$ such that $a \cdot [C] = 0$. Then $c([T_{C}], a) = 0$.
\end{enumerate}
In particular, {\rm (1)} and {\rm (2)} imply that the function $c$ induces a {\em crossed} homomorphism
\[ \mathcal C: \MCG(S) \rightarrow {\rm Hom} (H_{1}(S,\partial S),\Z) \simeq H^{1}(S;\Z); \qquad \phi \mapsto c([\phi], -). \]
\end{proposition}

\begin{proof}
First we prove (1).
Let $\phi N(a) \stackrel{\Sigma}{\rightarrow} N(\phi_{*}a)$ and $\phi N(a') \stackrel{\Sigma'}{\rightarrow} N(\phi_{*}a')$ be OB cobordisms. 
We place the surfaces $\Sigma$ and $\Sigma'$ so that 
\begin{itemize}
\item 
$\phi N(a)$ and $\phi N(a')$ in $S_{0}$ have the minimal geometric intersection (i.e., so do $N(a)$ and $N(a')$), and
\item 
$N (\phi_{*}a)$ and $N(\phi_{*}a')$ in $S_{1}$ have the minimal geometric intersection.
\end{itemize}

Let $H\subset S$ be one of the once-punctured torus chambers.  
Then $H \cap N(a)$ and $H\cap N(a')$ are oriented torus links. 
Suppose that $H \cap N(a)$ is a $(p,q)$-torus link and $H\cap N(a')$ is a $(p', q')$-torus link. 
Then $H\cap N(a+a')$ is a $(p+p', q+q')$-torus link.
By isotopy, we arrange the curves $H \cap N(a)$ and $H \cap N(a')$ realizing the minimum geometric intersection, hence in particular, they  transversely intersect.
We resolve all the intersection points as shown in Figure~\ref{fig:resolve}, and call the resulting multi-curve  $A_{H, a, a'}$. 
\begin{figure}[htbp]
\begin{center}
\includegraphics[width=60mm]{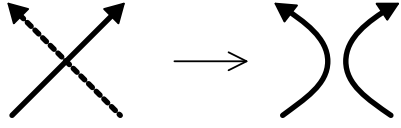}
\caption{Smoothing an intersection.}
\label{fig:resolve}
\end{center}
\end{figure}
Note that $[A_{H, a, a'}] = [H \cap N(a+a')]$ in $H_1(S, \partial S)$. 
We compare curves $A_{H, a, a'}$ and $H\cap N(a+a')$: 
\begin{enumerate}
\item[(i)]
Suppose that $(\sgn (p), \sgn(q))=(-\sgn (p'), -\sgn (q'))$. 
Let $n = \min\{ |p|, |q|, |p'|, |q'| \}$. 
Then $A_{H, a, a'}$ is the disjoint union of $H \cap N(a+a')$, $n$ circles bounding concentric discs oriented counterclockwise, and $n$ circles bounding concentric discs oriented clockwise. 
Removing the circles as shown in Figure~\ref{step12} yields $n$  negative and $n$ positive hyperbolic points. 
Hence we obtain an OB cobordism 
$$H\cap N(a+a') \stackrel{\Sigma_{H, a, a'}}{\longrightarrow} A_{H, a, a'}$$
with $d(\Sigma_{H, a, a'})=n-n=0.$
\item[(ii)]
Suppose that $(\sgn (p), \sgn(q)) \neq (-\sgn (p'), -\sgn (q'))$.  
In this case, we have $A_{H, a, a'} = H \cap N(a+a')$.
Hence we obtain a trivial OB cobordism 
$H\cap N(a+a') \stackrel{\Sigma_{H, a, a'}}{\longrightarrow} A_{H, a, a'}$
with $d(\Sigma_{H, a, a'})=0.$
\end{enumerate}
Next let $H \subset S$ be the $k$-th annulus chamber. 
Recall the properly embedded arc $\rho_k' \subset H$ joining the boundary circles $\gamma_0$ and $\gamma_k$ (cf. Figure~\ref{braid-generator}). 
Due to the definition of normal forms we may suppose that $H\cap N(a)=n \rho_k'$ and $H\cap N(a')=n' \rho_k'$. Then $H\cap N(a+a')= (n+n') \rho_k'$.
\begin{enumerate}
\item[(iii)] 
If $\sgn(n)=\sgn(n')$, let $A_{H, a, a'}:=(H\cap N(a))\sqcup (H\cap N(a')).$ Then $A_{H, a, a'} = H\cap N(a+a')$. 
Again we obtain a trivial OB cobordism 
$H\cap N(a+a') \stackrel{\Sigma_{H, a, a'}}{\longrightarrow} A_{H, a, a'}$
with $d(\Sigma_{H, a, a'})=0.$
\item[(iv)] 
If $\sgn(n)\neq\sgn(n')$, join $N(a)$ and $N(a')$ by describing arcs from the nearest pairs of $\rho_k'$ and $-\rho_k'$ to introduce $m=\min\{|n|, |n'|\}$ many hyperbolic singularities of the same sign $=\e$. 
See Figure~\ref{fig:resolution1}.
Call the resulting set of curves $A_{H, a, a'}$. 
Then $A_{H, a, a'}$ is the disjoint union of $H\cap N(a+a')$ and null-homologous nested arcs. 
This yields an OB cobordism $H\cap N(a+a') \stackrel{\Sigma_{H, a, a'}}{\longrightarrow} A_{H, a, a'}$
with $d(\Sigma_{H, a, a'})=\e m.$
\begin{figure}[htbp]
\begin{center}
\SetLabels
(.05*.6) $H$\\
\endSetLabels
\strut\AffixLabels{\includegraphics*[height=3cm]{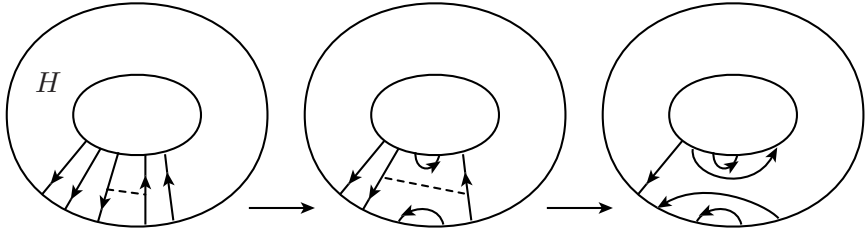}}
\caption{Case (iv). (Left) Curves $N(a) \sqcup N(a')$. (Right) $A_{H, a, a'}$.}
\label{fig:resolution1}
\end{center}
\end{figure}
\end{enumerate}

Let 
$$A_0 = \phi ( \bigsqcup_{H\subset S} A_{H, a, a'}), \qquad 
\Sigma_0 = (\phi \times {\rm id}) (\bigsqcup_{H\subset S} \Sigma_{H, a, a'}),$$ 
where the disjoint union is taken for all the $g+r-1$  chambers $H$ of $S$. 
Now we obtain an OB cobordism 
$$\phi N(a+a') \stackrel{\Sigma_0}{\longrightarrow} A_{0} \ \mbox{ with }\ d(\Sigma_0)= \sum_H d(\Sigma_{H, a, a'}).$$

We repeat the arguments parallel to (i)--(iv) by replacing $a$ by $\phi_*a$ and $a'$ by $\phi_*a'$. 
Namely, for each chamber $H$ we construct a multi-curve $A_{H, \phi_* a, \phi_*a'}$ from $H \cap (N(\phi_*a) \cup N(\phi_*a'))$ and obtain an OB cobordism 
$$A_{H, \phi_* a, \phi_*a'}  \stackrel{\Sigma_{H, \phi_*a, \phi_*a'}}{\longrightarrow} H\cap N(\phi_*(a+a')).$$
Let 
$$A_1 = \bigsqcup_{H\subset S} A_{H, \phi_*a, \phi_*a'}, \qquad 
\Sigma_1 = \bigsqcup_{H\subset S} \Sigma_{H, \phi_*a, \phi_*a'},$$ 
then we obtain an OB cobordism
$$A_1 \stackrel{\Sigma_1}{\longrightarrow} N(\phi_*(a+a'))\ \mbox{ with }\ d(\Sigma_1)= \sum_H d(\Sigma_{H, \phi_*a, \phi_*a'}).$$

\begin{claim}\label{claim:sign cancel}
We have $d(\Sigma_1)=-d(\Sigma_0)$.
\end{claim}

\begin{proof}
For cases (i, ii, iii), we have $d(\Sigma_{H, \phi_*a, \phi_*a'})=0$. 
For case (iv), i.e., $H$ is the $k$-th annulus chamber, since $\phi=$id  near $\partial S$, we have $H\cap N(\phi_*a)=H\cap N(a)= n \rho_k'$ and $H\cap N(\phi_*a')=H\cap N(a')= n' \rho_k'$.
Therefore, the OB cobordism $\Sigma_{H, \phi_*a, \phi_*a'}$ is given by the reverse direction as depicted in Figure~\ref{fig:resolution1}. 
Recalling that $d(\Sigma_{H, a, a'})=\e m$, we have $d(\Sigma_{H, \phi_*a, \phi_*a'})= -\e m$. This concludes the claim. 
\end{proof}

Next we construct an OB cobordism $A_{0} \stackrel{\Sigma^{+}}{\rightarrow} A_{1}$.
Recall that the OB-cobordism surfaces $\Sigma$ and $\Sigma'$ are obtained by sequence of configuration changes (cf. Figure~\ref{configure-change}). 
In general, a describing arc, $\delta$, of a hyperbolic singularity on $\Sigma$ may intersect $\Sigma'$ (or vice versa) as shown in the top left sketch of Figure~\ref{fig:resolve2}, where the black arc (resp. gray arcs) are leaves of $\Sigma$ (resp. $\Sigma'$), the dashed arc is $\delta$, and the dashed arrows indicate positive normal directions of the surfaces. 
\begin{figure}[htbp]
\begin{center}
\SetLabels
\endSetLabels
\strut\AffixLabels{\includegraphics*[width=100mm]
{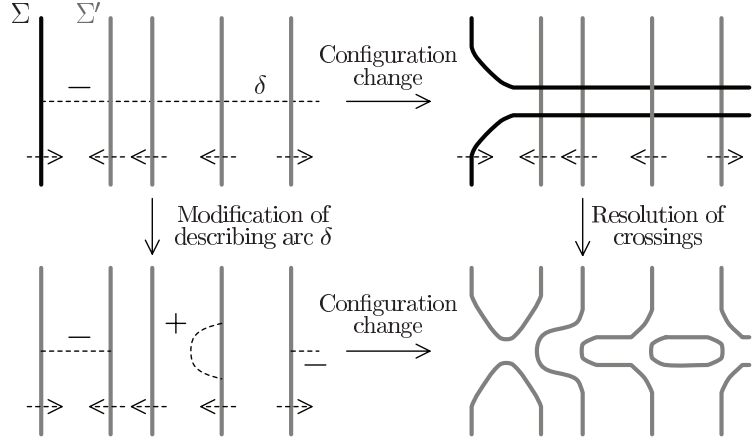}}
\caption{Modification of configuration changes.}
\label{fig:resolve2}
\end{center}
\end{figure}
By isotopy, we make $\delta, \Sigma, \Sigma'$ have no triple intersection points and $\delta$ and $\Sigma'$ attain the minimal geometric intersection.
We project $\delta$ to the diagram of $A_0$ then replace $\delta$ with several describing arcs for $A_0$ as in the vertical left passage in Figure~\ref{fig:resolve2} so that the diagram commutes.   
If the sign of original $\delta$ is $\e$, then the algebraic count of the replacing describing arcs is also $\e$.
This modification of configuration changes yields an OB cobordism $A_{0} \stackrel{\Sigma^{+}}{\rightarrow} A_{1}$.
By the construction of $\Sigma^{+}$, we have $d(\Sigma^{+}) = d(\Sigma) + d(\Sigma')$.

Finally we obtain a sequence of OB cobordisms 
\[ \phi N(a+a') \stackrel{\Sigma_{0}}{\longrightarrow} A_{0}  \stackrel{\Sigma^{+}}{\longrightarrow} A_{1} \stackrel{\Sigma_{1}}{\longrightarrow} N(\phi_{*}(a+a')).  \]
By Claim~\ref{claim:sign cancel}, 
\[ c([\phi],a+a') = d(\Sigma_{0}) + d(\Sigma^{+}) + d(\Sigma_{1}) = d(\Sigma) + d(\Sigma') = c([\phi], a) + c([\phi], a'). \]

We proceed to prove (2). 
Let $\phi N(a) \stackrel{\Sigma}{\rightarrow} N(\phi_{*}(a))$ be an OB cobordism. 
Extend $\psi \in \MCG(S, \partial S)$ to a diffeomorphism $\widetilde{\psi} =\psi \times id: S\times [0,1] \rightarrow S \times [0,1]$ and we obtain an OB cobordism $\psi \phi(N(a))  \stackrel{\widetilde\psi \Sigma}{\rightarrow} \psi( N(\phi_{*}(a)) )$.
Now let us take an  OB cobordism $\psi (N(\phi_{*}(a)) )  \stackrel{\Theta}{\rightarrow} N(\psi_{*} \phi_{*}(a))$.
Gluing $\widetilde\psi\Sigma$ and $\Theta$, we obtain an OB cobordism 
$\psi\phi(N(a))\stackrel{\widetilde\psi\Sigma}{\longrightarrow} \stackrel{\Theta}{\longrightarrow}N(\psi_{*}\phi_{*}(a))$. 
Since $\widetilde\psi$ preserves the signs and the number of hyperbolic singularities, $d(\Sigma) = d(\widetilde\psi\Sigma)$. This yields the desired equation.

To see (3), we observe that if a simple closed curve $C$ does not intersect the walls, then $T_{C}(N(a))$ is in the normal form for any $a \in H_1(S, \partial S)$, i.e., $T_C (N(a))= N(T_C a)$. 
Consider the product $\Sigma = T_{C}(N(a)) \times I$ which yields the trivial OB cobordism $T_C (N(a)) \stackrel{\Sigma}{\rightarrow} N(T_C a)$. Since the foliation is trivial, $c([T_{C}],a)=0$.

Finally, we prove (4). 
We construct an OB cobordism $T_{C}N(a) \stackrel{\Sigma}{\rightarrow} N(a) =N(T_C a)$ with $d(T_{C}N(a), N(a)) = 0$ as follows. 
Since $[C] \cdot [T_{C} N(a)] = [C] \cdot a= 0$,  by applying the configuration changes, described in Figure~\ref{configure-change}, to a portion of the multi-curve $T_{C} N(a)$ that lives in a small collar  neighborhood of $C$, we can modify $T_{C}N(a)$ so that it is disjoint from $C$. 
For example, the left sketch in Figure~\ref{fig:twists}
\begin{figure}[htbp]
\begin{center}
\includegraphics[width=100mm]{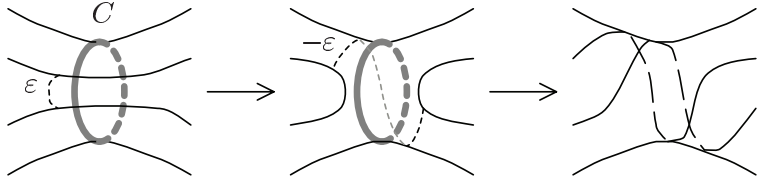}
\caption{Untwisting multi-curve $T_C N(a)$ (left) to obtain $N(a)$ (right).}
\label{fig:twists}
\end{center}
\end{figure}
depicts the case when the geometric intersection number $i(T_{C} N(a), C)=2$, where the thin dashed arcs indicate describing arcs for hyperbolic singularities.
Suppose that the sign of the hyperbolic singularity corresponding to the configuration change is $\e$. 
Next, we add a describing arc of sign $-\e$ to the deformed $T_C N(a)$  (cf. the middle sketch) so that the corresponding configuration change yields the multi-curve $N(a)$ (cf. the right sketch). 
This defines an OB cobordism $T_{C}N(a) \stackrel{\Sigma}{\rightarrow} N(a)$ which satisfies $d(T_{C}N(a), N(a)) = 0$.

When the geometric intersection number is greater than $2$, a similar construction applies. 
Especially the sum of the total algebraic count of the signs in the first operation and the second operation is $0$. 
\end{proof}

\subsection{The function $c$: Planar surface case}\label{sec:planar case}

\quad \\
In this section, we study the function $c$ for the case of $S=S_{0, r}$  a planar surface with $r$ boundary components.
We adopt the same notations as in Section~\ref{sec:sl-formula}.
The next proposition essentially has been proved in \cite{k} by direct analysis of the OB cobordism (though this terminology is not explicitly used). 
Based on the fact that $c$ is a crossed homomorphism we will give  more detailed expression of $c$.

Recall the arcs $\rho_j'$ and loops $\rho_j$ ($j=1, \dots, r-1$) specified in Figure~\ref{braid-generator}. 
Under Poincar\'e duality $H_1(S, \partial S; \Z) \simeq H^1(S; \Z)$; $[\rho_j'] \mapsto \PD [\rho_j']$, we may view $\{ [\rho_j'] \}_{j=1}^{r-1}$ as a basis of $H^1(S)$.
Let $\langle \cdot , \cdot \rangle$ denote the natural pairing of cohomology and homology.
Then we have $\langle [\rho_i'] , [\rho_j] \rangle = [\rho_i'] \cdot [\rho_j] = \delta_{i, j}$ the Kronecker delta.

\begin{proposition}\label{c-planar}
Let $S=S_{0, r}$ be a planar surface with $r$ boundary components. 
For $a \in H_1(S, \partial S;\Z)$ the function $c$ is formulated in the following way:
\begin{equation}\label{eq of c planar}
c([\phi], a)
= \sum_{i=1}^{r-1} \langle [\rho'_{i}], \phi_{*}a-a \rangle - \sum_{j=1}^{r-1} \langle a,[\rho_{j}] \rangle \langle [\rho'_{j}], \phi_{*}[\rho'_{j}]-[\rho'_{j}]\rangle 
\end{equation}
where $\phi_{*}a-a$ and $\phi_{*}[\rho'_{i}]-[\rho'_{i}]$ are regarded as elements of $H_{1}(S;\Z)$. 
Moreover, let $\{t_{i,j}\}_{1\leq i, j \leq r-1}$ be the matrix with $[\rho'_{i}] -\phi_{*}[\rho'_{i}] =\sum_{j=1}^{r-1} t_{i,j} [\rho_j]$ and suppose that $a = \sum_{j=1}^{r-1} x_j [\rho_j']$. 
Then {\rm (\ref{eq of c planar})} can be restated as follows. 
\begin{equation}\label{second formula}
c([\phi], a) = - \sum_{j=1}^{r-1}  x_j  \sum_{1\leq i \leq r-1,\; i\neq j} t_{j,i}
\end{equation}

\end{proposition}

\begin{remark}\label{rm:planar}
For the planar case, Proposition \ref{c-planar} shows that the crossed homomorphism $\mathcal C$ or $c([\phi], -)$ is completely determined by the map 
$\phi_{*}-id: H_1(S, \partial S) \rightarrow H_1(S)$.
\end{remark}

\begin{proof}
For $j=1, \ldots, r-1$, we have
\begin{eqnarray}\label{j'}
c([\phi], [\rho_j ' ]) = \sum_{1\leq i \leq r-1, \ i\neq j} \left\langle [\rho_i'],  \phi_* [\rho_j'] - [\rho_j']  \right\rangle
\end{eqnarray}
for the following reasons.
We recall that $c([\phi], [\rho_j'])$ counts algebraically the hyperbolic singularities produced by the configuration changes (cf. Figure~\ref{configure-change}) of the multi-curve $\phi(\rho_j')$ where it crosses the walls. 
We write $\phi$ as a product of special type of Dehn twists that are used in \cite{k} and denoted by $A_{k, l}, A_m$ there. 
We observe that a $\pm$ Dehn twist that involves the $i$-th and $j$-th binding components ($i\neq j$) contributes $\pm 1$ hyperbolic singularity for the OB cobordism $\phi(\rho_j') \stackrel{\Sigma}{\rightarrow} N(\phi_*[\rho_j'])$.
But a Dehn twist around a single binding component $\gamma_k$, where $(k=1, \ldots, r-1),$ does not contribute any hyperbolic singularity to the OB cobordism. 
Since the quantity $\langle [\rho_i'], \phi_*[\rho_j'] - [\rho_j'] \rangle$ counts algebraically the number of circles in $N(\phi_*[\rho_j'])$ around the binding $\gamma_i$, equation (\ref{j'}) follows.

Recall that $\{[\rho_{i}] \in H_1(S)\}_{i=1}^{r-1}$ is the dual basis of $\{[\rho'_{i}] \in H_1(S, \partial S)\}_{i=1}^{r-1}$. We may express $a \in H_1(S, \partial S) \simeq H^{1}(S)$ as:
$
a = \sum_j \langle a, [\rho_{j}] \rangle [\rho_j'].
$
By the crossed homomorphism property of the function $c$ (Proposition~\ref{property_of_c}), we can deduce (\ref{eq of c planar}) as follows:
\begin{eqnarray*}
c([\phi], a) 
&=&
\sum_j  \langle  a, [\rho_{j}] \rangle c([\phi], [\rho_j']) \\
&\stackrel{(\ref{j'})}{=}& 
\sum_{j} \sum_{i\neq j} 
\langle [\rho'_{i}], \langle a,[\rho_{j}]\rangle (\phi_{*} [\rho'_{j}] -[\rho'_{j}]) \rangle \\
& = & \sum_{i} \langle [\rho'_{i}], \phi_{*}a -a \rangle 
- \sum_{j} \langle a,[\rho_{j}] \rangle \langle [\rho'_{j}],\phi_{*}[\rho'_{j}]-[\rho'_{j}] \rangle.
\end{eqnarray*}

Now plugging the relation $[\rho'_j] -\phi_*[\rho'_j] =\sum_{i=1}^{r-1} t_{j, i} [\rho_i]$ to (\ref{j'}) we obtain 
\begin{equation}
c([\phi], [\rho_j']) = - \sum_{1\leq i \leq r-1,\; i\neq j} t_{j,i}. \label{jj'}
\end{equation}
Linearly extending (\ref{jj'}) to an arbitrary element $a =\sum_j x_j [\rho_j']$ we obtain (\ref{second formula}). 
\end{proof}

\begin{remark}
Since $\phi_*[\rho_j']=[\rho_j'] \in H_{1}(S,\partial S)$, we have $c([\psi], \phi_*[\rho_j'])= c([\psi], [\rho_j'])$. 
Therefore, when $S$ is planar the property (2) in Proposition~\ref{property_of_c} can be restated as
$$c([\psi \phi], [\rho_j']) = c([\phi], [\rho_j']) + c([\psi], [\rho_j']).$$
\end{remark}

By using Theorem \ref{theorem:sl-formula} and Proposition \ref{c-planar}, now we can deduce the self-linking number formulae in \cite{kp}, \cite{k}.
Let $a_{\sigma}$ (resp. $a_{\rho_j}$) be the exponent sum of the the braid generators $\{\sigma_{i}\}_{i=1}^{n-1}$ (resp. $\rho_j$) in the braid word $b=b_1^{\e_1} b_2^{\e_2} \cdots b_l^{\e_l}$.
Let $a = \sum_{i=1}^{r-1} s_{i}[\rho'_{i}] \in H_1(S, \partial S)$, the homology class introduced in the proof of Claim~\ref{claim:a} such that $[b] = a -\phi_{*}a$.

\begin{corollary}[The self-linking number formula for planar open books \cite{k}]\label{cor:sl-planar}
With the notations above, the self-linking number is given by the following formula. 
\begin{eqnarray*}
sl(\hat{b}, [\Sigma])
& = & -n + a_{\sigma} + \sum_{j=2}^{r} a_{\rho_{j}}(1-s_{j}) - \sum_{j=1}^{r-1} s_{j} \sum_{1\leq i \leq r-1 , \; i \neq j} t_{j,i}  
\end{eqnarray*}

\end{corollary}

\begin{proof}

Since $[b_{i}] \cdot [b_{j}]=0$ for all $b_{i}, b_{j}\in\{\rho_1, \ldots, \rho_{r-1}, \sigma_1, \ldots, \sigma_{n-1}\}$, we have 
$$\widehat{\exp}(b) =  \sum_{i=1}^l \e_i= a_{\sigma} + \sum_{j=1}^{r-1}a_{\rho_j},$$
and since $[\rho_j']\cdot[\rho_k] =\delta_{j,k}$ we have 
\[ \phi_{*}(a)\cdot[b]  = (a - [b])\cdot[b]   
= (\sum_{j=1}^{r-1} s_j [\rho_j'])\cdot (\sum_{i=1}^l \e_i [b_i] ) 
= \sum_{j=1}^{r-1} a_{\rho_{j}}s_{j}. \]
Hence by Theorem \ref{theorem:sl-formula} and Proposition \ref{c-planar}  we have:
\begin{eqnarray*}
sl(\hat{b},\Sigma) 
& = & -n + \widehat{\exp}(b) - \phi_*(a)\cdot[b] + c([\phi],a)\\
& = & -n + a_{\sigma} + \sum_{j=1}^{r-1} a_{\rho_{j}}(1-s_{j}) - \sum_{j=1}^{r-1} s_{j} \sum_{1\leq i \leq r-1 , \; i \neq j} t_{j,i}.   
\end{eqnarray*}
\end{proof}

\subsection{The function $c$: Surface with connected boundary}\label{sec:surface-one-boundary}

\quad\\
Let $S=S_{g, 1}$ be a genus $g$ surface with one boundary component.
When $g=1$, since there is no wall Proposition~\ref{property_of_c}-(3) implies that $c([\phi],a)=0$ for all $[\phi] \in \MCG(S_{1, 1})$ and $a \in H_1(S, \partial S)$.
Henceforth in this section we restrict our attention to the case $g\geq 2$.

We observe in the following example that, unlike the planar case discussed in Remark~\ref{rm:planar}, the function $c([\phi],-)$ is no longer completely determined by the action of $\phi_*$ on homologies.  
In fact we see in Proposition~\ref{c-genus} that $c([\phi],-)$ carries more delicate information of $[\phi] \in \MCG(S)$.

\begin{example}\label{example C and C'}
Let us take simple closed curves $C, C'$ and $\rho$ as in Figure~\ref{example_twist}.
\begin{figure}[htbp]
\begin{center}
\SetLabels
(.7*.2) $C$\\
(.8*1) $C'$\\
(.85*.5) $\rho$\\
(.6*.05) wall\\
\endSetLabels
\strut\AffixLabels{\includegraphics*[width=8cm]{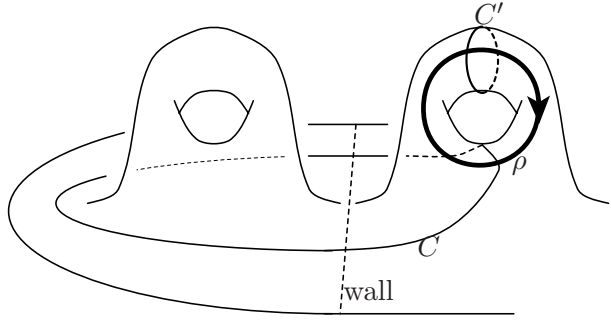}}
\caption{Curves $C, C'$ and $\rho$ in Example~\ref{example C and C'}.}
\label{example_twist}
\end{center}
\end{figure}
Since $C$ and $C'$ cobound a subsurface, $T_{C}$ and $T_{C'}$ induce the same action on the homology groups $H_{1}(S;\Z)$ and $H_{1}(S,\partial S;\Z)$. 
As shown in Figure \ref{fig:example_change}, we modify the curve $T_{C}(\rho)$ into the normal form $N(T_{C}(\rho))$ by introducing three positive hyperbolic singularities and one negative hyperbolic singularity. 
Hence $c([T_{C}],[\rho]) = 3-1= 2$.
On the other hand, $C'$ does not intersect the walls, so by Proposition~\ref{property_of_c}-(3) we get $c([T_{C'}], [\rho])=0$.
\begin{figure}[htbp]
\begin{center}
\SetLabels
(0.41*0.6)   $T_{C}(\rho)$\\
(0.22*0.65)   $+$\\
(0.63*0.7)   $+$\\
(0.13*0.32)  $+$\\
(0.77*0.01)  $-$\\
\endSetLabels
\strut\AffixLabels{\includegraphics*[scale=0.5, width=120mm]{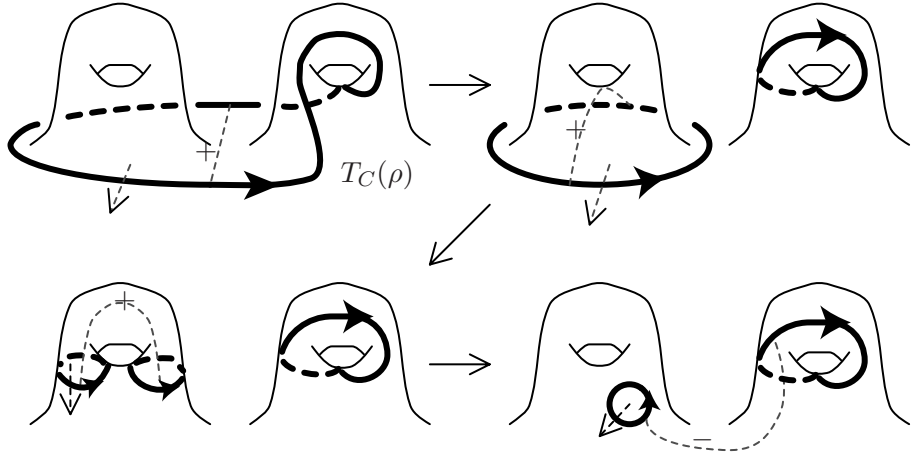}}
 \caption{Configuration change of $T_{C}(\rho)$ to the normal form $N(T_C (\rho))$.}
  \label{fig:example_change}
\end{center}
\end{figure}

\end{example}

In this section we use the following notations: 
Recall the circles $\rho_j, \rho_j' \subset S$ defined in Section~\ref{sec:sl-formula}. 
To distinguish elements of $H_{1}(S;\Z)$ and $H_{1}(S,\partial S;\Z) \cong H^{1}(S;\Z)$, we use the symbol $[\rho_{j}]$ ($j=1, \ldots, 2g$) to express the homology class of $H_{1}(S)$ represented by the circle $\rho_{j}$, and the symbol $[\rho'_{j}]$ for the relative homology class of $H_{1}(S, \partial S) \cong H^1(S)$ represented by the circle $\rho'_{j}$. 
Note that since $S$ has connected boundary, $\rho_{j} = \rho'_{j}$ as a set for all $j$ and as a group $H_{1}(S;\Z) \cong H_{1}(S,\partial S;\Z) \cong \Z^{2g}$.

Let $\langle\;,\;\rangle : H_{1}(S,\partial S ;\Z) \times H_{1}(S;\Z) \to \Z$ denote the natural pairing of cohomology and homology, or the intersection pairing, i.e., $\langle \PD[\rho_j'], [\rho_k]\rangle = [\rho_j'] \cdot [\rho_k]$. 
For simplicity, we denote $\PD[\rho_j']$ by $[\rho_j']$ in the following.  
We have:
$$\langle [\rho_j'], [\rho_k]\rangle = \left\{
\begin{array}{ll}
1 & \mbox{ if } (j, k)=(2i-1, 2i), \\
-1 & \mbox{ if } (j, k)=(2i, 2i-1), \\
0 & \mbox{ otherwise. }
\end{array} \right. 
$$

Let $\Gamma_1 :=\pi_{1}(S)$ the fundamental group, $\Gamma_2:=[\Gamma_1, \Gamma_1]$ the commutator subgroup, and $\Gamma_{k+1} := [\Gamma_k, \Gamma_1]$, namely $\{\Gamma_k\}_{k\geq 1}$ is the the lower central series of $\Gamma_1$.  
Then the natural action of $\MCG(S_{g, 1})$ on $\Gamma_1$ induces the $k$-th {\em Johnson-Morita representation} \cite[p.199]{mo3}
\[ 
\varrho_k: \MCG( S_{g,1}) \rightarrow \textrm{Aut}(\Gamma_1 \slash \Gamma_k), \quad k\geq 2.
\]
Let $\M(k):=\ker \rho_k$ and $H:=H_1(S; \Z)$.  
Morita generalizes the Johnson homomorphism $\tau_2: \M(2) \to \Hom(H, \Gamma_2/ \Gamma_3)$ to the $k$-th {\em Johnson-Morita homomorphism} \cite[p.201]{mo3}: 
$$
\tau_k: \M(k) \to \Hom(H, \Gamma_k/\Gamma_{k+1})
$$
with $\ker\tau_k = \M(k+1)$. 
Let $\mK$ be the subgroup of $\MCG(S_{g,1})$ generated by the Dehn twists about separating simple closed curves in $S$.
Johnson proves in \cite{j} that for $g\geq 3$ we can identify $\ker\tau_2 = \mK$. 
Recall that by Proposition~\ref{property_of_c}-(2, 4) our crossed homomorphism $\mathcal C: \MCG(S)\to H^1(S;\Z)$ also vanishes on $\mathcal K$. 
Hence it is natural to expect that the map $\mathcal C$ is related to $\tau_2$. 

Associated to the representation $\varrho_3: \MCG(S_{g,1}) \to \Aut(\Gamma_1/\Gamma_3)$ 
Morita \cite{mo3} finds the embedding $\MCG(S_{g,1})/\M(3) \subset \frac{1}{2} \wedge^3 H \rtimes Sp(H)$ as a finite index subgroup and the crossed homomorphism $\tilde k: \MCG(S_{g,1}) \to \frac{1}{2} \wedge^3 H$, which is the unique (modulo coboundaries for $H^1(\MCG(S_{g,1}), \wedge^3 H)$) extension of $\tau_2$. 
For our purpose we are interested in the composition 
\[ k=C \circ \tilde k: \MCG(S_{g,1}) \rightarrow H^{1}(S;\Z)\]
where $C: \frac{1}{2} \wedge^3 H \to H$ is the contraction defined by $C(x \wedge y \wedge z) = 2[ (x \cdot y) z + (y \cdot z) x + (z \cdot x) y]$. 
The associated map, which we denote by the same letter, $k:\MCG(S_{g,1})\times H_{1}(S) \rightarrow \Z$ given by $k(\phi,a) = k(\phi)(a)$ is a crossed homomorphism.
Since $k$ is a generator of the cohomology group $H^{1}(\MCG(S_{g,1}) ; H) \cong \Z$ \cite[Rem 4.9]{mo3}, it is natural to expect that $k$ appears in the description of $c(\phi,a)$.

Below we fix conventions and define the crossed homomorphism $k$ following Morita's \cite[\S 6]{mo} that is based on combinatorial group theory.

\begin{definition}
Let $F_{2}$ be the free group of rank two with generators $\alpha$ and $\beta$.
Any element of $F_2$ is uniquely written in the form $\alpha^{\epsilon_{1}} \beta^{\delta_{1}} \cdots \alpha^{\epsilon_{n}} \beta^{\delta_{n}},$ where $\epsilon_i, \delta_i \in \{-1, 0, 1\}$. 
With this expression, we define a function $d: F_2 \to \Z$ by 
\[ d(\alpha^{\epsilon_{1}} \beta^{\delta_{1}} \cdots \alpha^{\epsilon_{n}} \beta^{\delta_{n}}) = \sum_{i=1}^n \delta_i \sum_{j=1}^i \epsilon_j.  \]
Let $\alpha_i, \beta_i$ ($i=1, \ldots, g$) be generating curves of $\pi_1(S)$ as in Figure~\ref{generator-Morita}.
\begin{figure}[htbp]
\begin{center}
\SetLabels
(.15*.3) $\alpha_1$\\
(.37*.3) $\alpha_2$\\
(.65*.25) $\alpha_g$\\
(.2*.5) $\beta_1$\\
(.5*.3) $\beta_2$\\ 
(.85*.2) $\beta_g$\\
\endSetLabels
\strut\AffixLabels{\includegraphics*[width=10cm]{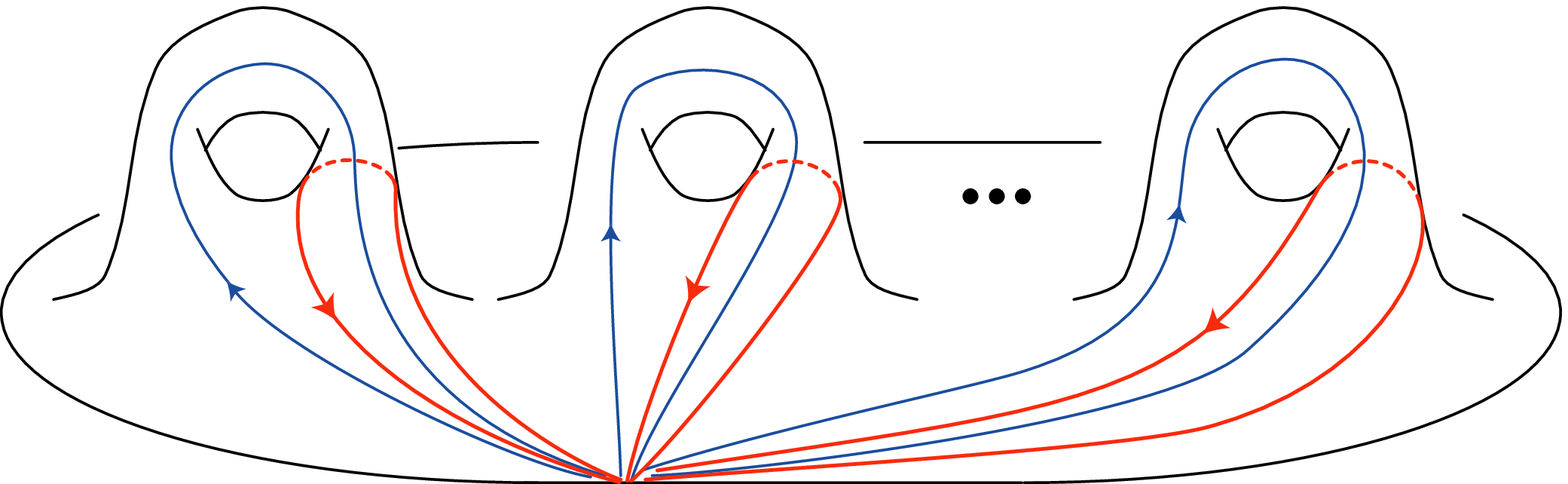}}
\caption{}
\label{generator-Morita}
\end{center}
\end{figure}
Let $p_{i}: \pi_{1}(S) \rightarrow F_2$ be a homomorphism defined by 
\[ 
p_{i}(\gamma) = \left\{ 
\begin{array}{ll}
 \alpha & \;\;\; \mbox{ if } \gamma = \alpha_i, \\
 \beta & \;\;\; \mbox{ if }  \gamma = \beta_i,\\
 1 & \;\;\; \mbox{ otherwise}.
\end{array} 
\right.
\]
Finally we define a map $k: \MCG(S_{g, 1}) \times H_{1}(S,\partial S) \rightarrow \Z$ by
\[ k([\phi], a) = \sum_{i=1}^{g} d(p_{i}(\phi_{*}{\tt a})) - d(p_{i}({\tt a})) \]
where ${\tt a} \in \pi_{1}(S)$ represents $a \in H_1(S, \partial S)$.
Morita proves in \cite[Lemma 6.3]{mo} that $k([\phi], a)$ is a crossed homomorphism.
\end{definition}

We give an explicit formula of the function $c$ by using $k$. 
It provides a new geometric meaning of the classically known crossed homomorphism $k$: the signed count of the saddle points in an OB cobordism.

\begin{proposition}\label{c-genus}
If $S=S_{g,1}$ has connected boundary and $g\geq 2$, then the function $c$ is expressed as
\begin{eqnarray}\label{formulac}
c([\phi],a) &=&
-2 k([\phi],a) + \sum_{i=1}^{g} 
\langle [\rho_{2i-1}'] - [\rho_{2i}'], \ \phi_* a-a \rangle \\
&& - \sum_{i=1}^g 
\langle a, [\rho_{2i-1}] \rangle 
\langle [\rho_{2i}'], \ \phi_*[\rho_{2i}']-[\rho_{2i}'] \rangle \notag \\
&& -  \sum_{i=1}^g 
\langle a, [\rho_{2i}] \rangle 
\langle [\rho_{2i-1}'], \ \phi_*[\rho_{2i-1}']-[\rho_{2i-1}'] \rangle \notag
\end{eqnarray}
where $\phi_{*}a-a$ and $\phi_{*}[\rho_{j}']-[\rho'_{j}]$ are regarded as  elements of $H_{1}(S; \Z)$.
\end{proposition}

\begin{proof}

Recall that the left hand side of (\ref{formulac}) satisfies the crossed homomorphism properties (1), (2) in Proposition~\ref{property_of_c}.
Hence it is sufficient to verify (\ref{formulac}) for a generating set of the mapping class group $\MCG(S_{g, 1})$.

We use the {\em Lickorish generators} of $\MCG(S_{g,1})$. 
Let $A_{i},B_{i}$ $(i=1,\ldots,g)$ and $C_{j}$ $(j=1,\ldots,g-1)$ be simple closed curves as shown in Figure~\ref{fig:generator}.
\begin{figure}[htbp]
\begin{center}
\SetLabels
(0.82*1.0)  $A_{1}$\\
(0.88*0.38)  $B_{1}$\\
(0.75*0.13)  $C_{1}$\\
(0.55*1.0)  $A_{2}$\\
(0.6*0.38)  $B_{2}$\\
(0.5*0.18)  $C_{2}$\\
(0.2*1.0)  $A_{g}$\\
(0.15*0.38)  $B_{g}$\\
(0.25*0.18)  $C_{g-1}$\\
\endSetLabels
\strut\AffixLabels{\includegraphics*[scale=0.5, width=110mm]{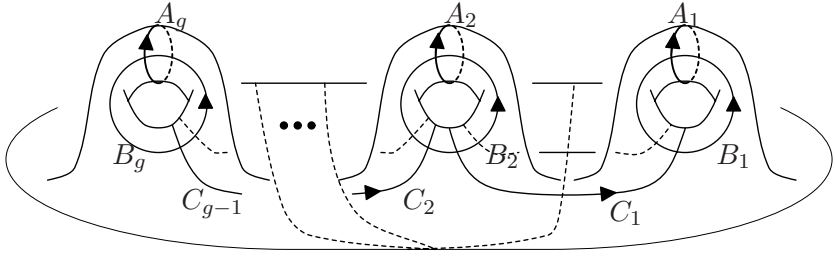}}
\caption{Generating curves for $\MCG(S_{g, 1})$.}\label{fig:generator}
\end{center} 
\end{figure}
Lickorish proved that the Dehn twists along these $3g-1$ curves generate $\MCG(S_{g, 1})$.
With the orientations indicated in Figure~\ref{fig:generator}, we have in $H_{1}(S;\Z)$ that 
$$[A_{i}] = [\rho_{2i-1}], \ [B_{i}]=[\rho_{2i}], \mbox{ and }[C_{i}] = -[\rho_{2i-1}]+[\rho_{2i+1}].$$

If $D \in \{A_{i},B_{i},C_{i}\}$ is disjoint from the loop $\rho_{j}$, then $c([T_{D}], [\rho'_{j}])= k([T_{D}], [\rho'_{j}]) = 0$ and ${T_{D}}_*[\rho'_{j}]-[\rho'_{j}] = 0$, thus the formula (\ref{formulac})  holds.
So we only need to consider the case where $D$ has non-trivial intersection with $\rho_{j}$. There are four cases to study:\\

{\bf Case I:} $(\phi, a) = (T_{A_i}, [\rho_{2i}'])$ \\
Since $A_{i}$ is disjoint from the walls, Proposition~\ref{property_of_c}-(3) implies that $c([T_{A_{i}}],[\rho'_{2i}]) = 0$. 
On the other hand, 
$T_{A_{i}}(\rho_{2i}) = \rho_{2i-1}\rho_{2i} = \beta_{g-i} \alpha_{g-i}^{-1}$ in $\pi_1(S)$ hence $k([T_{A_{i}}], [\rho'_{2i}]) = d(\beta \alpha^{-1})=0$.
Finally observe that ${T_{A_i}}_* [\rho'_{2i}] -[\rho'_{2i}] = [\rho_{2i-1}]$, hence 
\begin{eqnarray*}
(\star) 
&:=& \sum_{k=1}^{g} \langle [\rho_{2k-1}'] - [\rho_{2k}'], \ \phi_* a-a \rangle 
= \langle -[\rho_{2i}'], [\rho_{2i-1}] \rangle =1,\\
(\star \star) 
&:=& \sum_{k=1}^g 
\langle a, [\rho_{2k-1}] \rangle 
\langle [\rho_{2k}'], \ \phi_*[\rho_{2k}']-[\rho_{2k}'] \rangle 
+  \sum_{k=1}^g 
\langle a, [\rho_{2k}] \rangle 
\langle [\rho_{2k-1}'], \ \phi_*[\rho_{2k-1}']-[\rho_{2k-1}'] \rangle \\
&=& 
\langle [\rho_{2i}'], [\rho_{2i-1}] \rangle^2  = (-1)^2 = 1. 
\end{eqnarray*}
Thus the equality (\ref{formulac}) holds.\\

{\bf Case II:} $(\phi, a) = (T_{B_i}, [\rho_{2i-1}'])$ \\ 
As in the Case I, $B_{i}$ is disjoint from the walls, so $c([T_{B_i}], [\rho'_{2i-1}])= 0$. 
On the other hand, ${T_{B_i}} (\rho_{2i-1}) = \rho_{2i-1}\rho_{2i}^{-1}= \beta_{g-i} \alpha_{g-i}$, hence $k([T_{B_{i}}], [\rho'_{2i-1}]) = d(\beta \alpha)=0$.
Finally observe that ${T_{B_i}}_*[\rho'_{2i-1}] -[\rho'_{2i-1}] = -[\rho_{2i}]$, hence 
\begin{eqnarray*}
(\star) &=& 
\langle [\rho_{2i-1}'],  -[\rho_{2i}] \rangle = -1, \\
(\star \star) &=&  
\langle [\rho_{2i-1}'], [\rho_{2i}] \rangle
\langle [\rho_{2i-1}'], -[\rho_{2i}] \rangle  = -1. 
\end{eqnarray*}
Thus the equality (\ref{formulac}) holds. \\

{\bf Case III:} $(\phi, a)=(T_{C_{i}}, [\rho_{2i}'])$ \\
Observe that $c([T_{C_{i}}], [\rho'_{2i}])= -1$.
Since 
$$T_{C_{i}}(\rho_{2i}) = \rho_{2i}\rho_{2i+1}^{-1}\rho_{2i}^{-1}\rho_{2i-1}\rho_{2i} = \alpha_{g-i}^{-1} \beta_{g-i-1}^{-1} \alpha_{g-i} \beta_{g-i} \alpha_{g-i}^{-1},$$ 
$k([T_{C_i}], [\rho'_{2i}]) = d(\beta \alpha^{-1}) + d(\beta^{-1})=0$.
Finally, ${T_{C_i}}_*[\rho'_{2i}]-[\rho'_{2i}] = [\rho_{2i-1}] - [\rho_{2i+1}]$, hence 
\begin{eqnarray*}
(\star) &=& 0, \\
(\star \star) &=&  
\langle [\rho_{2i}'], [\rho_{2i-1}] \rangle
\langle [\rho_{2i}'], -[\rho_{2i+1}] + [\rho_{2i-1}] \rangle  = (-1)^2=1. 
\end{eqnarray*}
Thus the equality (\ref{formulac}) holds.\\

{\bf Case IV:} $(\phi, a) = (T_{C_{i}}, [\rho_{2i+2}'])$\\
In this case, $c([T_{C_{i}}], [\rho'_{2i+2}])= 1$
and $$T_{C_{i}}(\rho'_{2i+2}) = \rho_{2i}^{-1}\rho_{2i-1}^{-1}\rho_{2i}\rho_{2i+1}\rho_{2i+2} = \alpha_{g-i} \beta_{g-i}^{-1} \alpha_{g-i}^{-1} \beta_{g-i-1} \alpha_{g-i-1}^{-1}.$$
Hence $k([T_{C_{i}}], [\rho'_{2i+2}]) = d(\alpha \beta^{-1} \alpha^{-1}) + d(\beta \alpha^{-1}) = -1$.
Finally, ${T_{C_i}}_{*}[\rho'_{2i+2}] -[\rho'_{2i+2}] = -[\rho_{2i-1}] +[\rho_{2i+1}]$, hence 
\begin{eqnarray*}
(\star) &=& 0, \\
(\star \star) &=&  
\langle [\rho_{2i+2}'], [\rho_{2i+1}] \rangle
\langle [\rho_{2i+2}'], -[\rho_{2i-1}] + [\rho_{2i+1}] \rangle  = (-1)^2=1. 
\end{eqnarray*}
Thus the equality (\ref{formulac}) holds.
These computations complete the proof.
\end{proof}

The map $k$ appears in various contexts in the theory of mapping class groups (see \S 2 of \cite{mo2} for concise overview). In particular, $k$ can be interpreted in terms of winding numbers of curves on surfaces. Fixing a non-vanishing vector field $X$ on $S$, one defines the {\em winding number} of an oriented simple closed curve $\gamma$ on $S$ as the rotation number of the tangent vector to $\gamma$ with respect to $X$ as $\gamma$ is traversed once positively. Then $k(\phi,\gamma)$ is equal to the difference of winding numbers of $\phi(\gamma)$ and $\gamma$ as stated in  Def.1.3.1 of Trapp's paper \cite{tra}.

Recall that in (Step 1) near Figure~\ref{step12} we have observed that a c-circle bounding a disc contributes $\pm 1$ to the function $c([\phi], a)$. Such a disc also contributes $\pm1$ to the above winding number. 

In addition, the self-linking number $sl(\gamma, [\Sigma])$ is the winding number of a nowhere vanishing section $X$ of the vector bundle $\xi|_\Sigma \to \Sigma$ along $\gamma$ relative to $\Sigma$, where $\Sigma$ is a Seifert surface of $\gamma$.

Interestingly, the keyword of the above facts is ``winding number''. 
The authors thank the anonymous referee for pointing this out.

Theorem~\ref{theorem:sl-formula} and Proposition~\ref{c-genus} give a new relationship between the contact structures of $3$-manifolds  and the Johnson-Morita homomorphisms.
This develops into the following question:
Our result roughly says that if we choose a homology class $a \in \Gamma_1/\Gamma_2= H_{1}(S;\Z)$ (from geometric point of view, this choice corresponds to the choice of Seifert surface of the transverse link $L=\widehat{b}$), then the Johnson-Morita representation $\varrho_3:\MCG(S_{g,1}) \to {\rm Aut}(\Gamma_1/\Gamma_3)$ gives the self-linking number. 
Now we ask whether a similar phenomenon occurs for the higher Johnson-Morita representation $\varrho_i: \MCG(S_{g, 1}) \to {\rm Aut}(\Gamma_1/\Gamma_{i})$, where $i > 3$, and provides new invariants of transverse links?

\subsection{The function $c$: General surface case} 

\quad\\
Finally we give a complete description of the function $c$ for general surfaces $S=S_{g, r}$.
We use the same convention as in Section~\ref{sec:surface-one-boundary}, that is, $[\rho'_{j}]$ is an element of $H_{1}(S,\partial S;\Z) \cong H^1(S; \Z)$ and $[\rho_{j}]$ is an element of $H_{1}(S;\Z)$.
Let $S' = S_{g, 1}$ be the surface obtained from $S=S_{g, r}$ by filling the boundary components  
$\gamma_1, \dots, \gamma_{r-1}$ by discs and $i: S \rightarrow S'$ the canonical inclusion. 
Let $\pi: \MCG(S_{g, r}) \rightarrow \MCG(S_{g, 1})$ be the forgetful map.
Let us consider the pull-back $\pi^{*}k : \MCG(S_{g, r}) \times H_{1}(S,\partial S) \rightarrow \Z$ of the crossed homomorphism $k$ defined by
\[ \pi^{*}k : ([\phi], a) \mapsto k(\pi[\phi], i_{*}(a) ). \]
For $1 \leq i \leq r-1+2g$, let 
$$
[\varsigma_j'] = \left\{
\begin{array}{ll}
[\rho_j'] & \mbox{ if } j=1, \ldots, r-1, \\
- [ \rho'_{j+1}] & \mbox{ if } i= r, r+2, \ldots, r-2+2g, \\
{ } [ \rho_{j-1}^{'}]  & \mbox{ if } i=r+1, r+3, \ldots, r-1+2g.
\end{array}
\right. 
$$
In particular, we have $\langle [\varsigma_i'], [\rho_j] \rangle = \delta_{i, j}$.
By combining Propositions \ref{c-planar} and \ref{c-genus} we get an explicit formula of the function $c$.

\begin{theorem}[A formula of function $c$]
\label{theorem:c-formula}

Let $S=S_{g,r}$ be the surface with genus $g$ and $r$ boundary components.
The function 
$c: \MCG(S_{g, r}) \times H_{1}(S,\partial S;\Z) \rightarrow \Z$ has the following expression: 
\[ c([\phi],a) = -2 (\pi^{*}k)([\phi], a) 
+ \sum_{j=1}^{2g+r-1} \langle [\varsigma_{j}'], \phi_{*}a-a\rangle 
- \sum_{j=1}^{2g+r-1}\langle a, [\rho_{j}] \rangle 
\langle [\varsigma'_{j}], \phi_{*}[\varsigma'_{j}] -[\varsigma'_{j}]\rangle \]
where $\phi_{*}a-a$ and $\phi_{*}[\varsigma'_{j}] -[\varsigma'_{j}]$ are regarded as elements of $H_{1}(S;\Z)$.
\end{theorem}

\section{On the Bennequin-Eliashberg inequality}\label{sec:OT-disc}

In this section using open book foliations we give a new proof to the Bennequin-Eliashberg inequality \cite{el2}.

Recall that an {\em overtwisted disc} is an embedded disc whose boundary is a limit cycle in the characteristic foliation. Thus an overtwisted disc always has {\em Legendrian} boundary. 
As a corresponding notion in the framework of open book foliations we introduce the following:

\begin{definition}\label{def:trans-ot-disc}
Let $D \subset M_{(S, \phi)}$ be an oriented disc whose boundary is a positively braided unknot.
If the following are satisfied $D$ is called a {\em transverse overtwisted disc}: 
\begin{enumerate}
\item $G_{--}$ (Def \ref{def:negativity-graph}) is a connected tree with no fake vertices.
\item $G_{++}$ is homeomorphic to $S^1$.
\item $\F(D)$ contains no c-circles. 
\end{enumerate}
\end{definition}

By Proposition~\ref{sl-formula-1} we observe that $sl(\partial D, [D]) = 1$ for a transverse overtwisted disc $D$. 

\begin{proposition}\label{prop:overtwisted disc}
If $(S, \phi)$ contains a transverse overtwisted disc then the compatible  contact 3-manifold $(M,\xi)$ contains an overtwisted disc. 
\end{proposition}

\begin{proof}
By Theorem~\ref{identity theorem} and applying Giroux's elimination lemma (see \cite[p.187]{geiges}) we can convert a transverse overtwisted disc to an overtwisted disc.
\end{proof}

We will prove the converse in Corollary~\ref{cor:converse}, hence the existence of a transverse overtwisted disc is equivalent to the existence of a usual overtwisted disc.

\begin{theorem}[The Bennequin-Eliashberg inequality \cite{el2}]
\label{theorem:BEinequality}
If a contact 3-manifold $(M,\xi)$ is tight, then for any null-homologous transverse link $L$ and its Seifert surface $\Sigma$, the following inequality holds: 
\[ sl(L,[\Sigma]) \leq -\chi(\Sigma) \]
\end{theorem}

The following corollary was pointed out by John Etnyre and a proof is straightforward. 

\begin{corollary}
The following are equivalent: 

\begin{enumerate}
\item
$(M,\xi)$ is tight.
\item
For any null-homologous transverse link $L$ and its Seifert surface $\Sigma$ we have $sl(L,[\Sigma]) \leq -\chi(\Sigma)$.
\item
For any transverse unknot $U=\partial D$ we have $sl(U,[D]) \leq -\chi(D)=-1$. 
\end{enumerate} 
\end{corollary}

We use the following Lemma \ref{lemma:keylemma} and Proposition~\ref{sl-formula-1} to prove Theorem~\ref{theorem:BEinequality}.

\begin{lemma}\label{lemma:keylemma}
Let $L$ be a null-homologous transverse link in a contact 3-manifold $(M,\xi)$ and $\Sigma$ be a Seifert surface for $L$. 
Assume that 
\[ sl(L,[\Sigma]) > -\chi(\Sigma), \]
that is, the Bennequin-Eliashberg inequality is violated.
With some perturbation of $\Sigma$ fixing the boundary we can make 
the graph $G_{--}$ contain a contractible component with no fake vertices. 
\end{lemma}

\begin{proof}
Using Propositions~\ref{poincare-hopf} and \ref{sl-formula-1}, we assume that $sl(L, [\Sigma]) + \chi(\Sigma) = 2(e_{-} - h_{-}) >0$, i.e., $e_{-}-h_{-} > 0$.
Let $\Gamma_{1}, \ldots ,\Gamma_{k}$ denote the connected components of the graph $G_{--}$.
Let $f(\Gamma_{i})$ be the number of the fake vertices of $\Gamma_{i}$ and $e_{-}(\Gamma_{i})$ the number of the negative elliptic points in $\Gamma_i$. 
Let $h_{-}(\Gamma_{i})$ be the number of the edges in $\Gamma_{i}$.
By Proposition~\ref{prop:no-c-circle} with some perturbation of $\Sigma$ fixing the boundary we may assume that $\F(\Sigma)$ has no $c$-circles, hence the region decomposition (Proposition~\ref{prop:region}) does not contain type $ac$, $bc$ or $cc$ regions, so $e_- = \sum_{i=1}^k e_-(\Gamma_i)$ and $h_- = \sum_{i=1}^k h_-(\Gamma_i)$. 
Since $\Gamma_{i}$ is connected, the Euler characteristic of $\Gamma_i$ satisfies that:
\[ \chi(\Gamma_{i}) = ( f( \Gamma_{i}) + e_{-}(\Gamma_{i}) ) - h_{-}(\Gamma_{i}) \leq 1, \]
i.e., 
$e_{-}(\Gamma_{i} )- h_{-}(\Gamma_{i}) \leq 1-f(\Gamma_{i}).$
Therefore we obtain that $e_{-}(\Gamma_{i}) - h_{-}(\Gamma_{i}) = 1$ if and only if $f(\Gamma_{i})=0$ and $\Gamma_{i}$ is contractible. 
Now we have: 
\[ 
0 < e_{-}-h_{-} 
= \sum_{i=1}^k e_{-}(\Gamma_{i}) - \sum_{i=1}^k h_{-}(\Gamma_{i}) = \sum_{i=1}^k (e_{-}(\Gamma_{i}) - h_{-}(\Gamma_{i})) \leq \sum_{i=1}^k(1-f(\Gamma_{i})). 
\]
Thus for some $i$, the equality $e_{-}(\Gamma_{i})- h_{-}(\Gamma_{i}) = 1$ must hold, which implies that $\Gamma_{i}$ is contractible and has no fake vertices.
\end{proof}

Now we are ready to prove Theorem~\ref{theorem:BEinequality}.
Eliashberg's original proof to the Bennequin-Eliashberg inequality uses characteristic foliation theory. 
We give an alternative proof from a view point of open book foliations.

\begin{proof}[Proof of Theorem~\ref{theorem:BEinequality}]

Suppose that there exists a null-homologous transverse link $L$ in $(M,\xi)$ with a Seifert surface $\Sigma$ such that $sl(L, [\Sigma])> -\chi(\Sigma)$. We will show that $\xi$ is overtwisted.

Fix an open book $(S,\phi)$ which supports $\xi$ and isotope $L$ and $\Sigma$ with the transverse link type of $L$ preserved until it admits an open book foliation $\F(\Sigma)$.
By Proposition~\ref{prop:no-c-circle} and Lemma~\ref{lemma:keylemma}, we may assume that $\F(\Sigma)$ contains no $c$-circles and the negativity graph $G_{--}$ contains a contractible component $\Gamma \subset G_{--}$ with no fake vertices. 
In particular, the induced region decomposition of $\Sigma$ consists only of $aa$-, $ab$-, and $bb$-tiles.

Let $\mathcal R \subset \Sigma$ be the set of $b$-arcs that end on the vertices of $\Gamma$.
Since $\Gamma$ lives only in $ab$- and $bb$-tiles and has no fake vertices, we have $\Gamma \subset \Int(\overline{\mathcal R})$, where $\overline{\mathcal R}$ is the closure of $\mathcal R$. 
Let $\mathcal P =G_{++}(\overline{\mathcal R})$ be the set of positive elliptic points in $\overline{\mathcal R}$ and the stable separatrices approaching to the positive hyperbolic points in $\overline{\mathcal R}$. 
Since $\Gamma$ is a tree with no fake vertices, $\overline{\mathcal R}\smallsetminus \mathcal P$ is an open disc, $D$, embedded in $\Sigma$.

In general, $\overline{\mathcal R}$ may not be a disc, or $\partial D= \mathcal P$ may not be an embedded circle.  
Let $\mathcal P_\circ \subset \mathcal P$ denote the subset of $\mathcal P$ where we cut out $\overline{\mathcal R}$ to obtain $D$. 
We have $\overline{\overline{\mathcal R} \setminus \mathcal P_\circ} = \overline D$. 
A connected component $\lambda$ of $\mathcal P_\circ$ is either 
\begin{itemize}
\item[(i)] a positive elliptic point like in Figure~\ref{fig:neighbor-lambda}, or 
\item[(ii)] a union of stable separatrices like the thick arcs in Figures~\ref{fig:neighbor-lambda2} and \ref{fig:neighbor-lambda3}. 
\end{itemize}
\begin{figure}[htbp]
\begin{center}
\SetLabels
(.6*.8) part of $\overline{\mathcal R}$\\
(0*.8) part of $\overline{\mathcal R}$\\
(.82*.8) part of $\widetilde D$\\
(.21*.2) $\lambda$\\
(.6*.5) $\lambda$\\
(.92*.7) $\lambda_1$\\
(.92*.31) $\lambda_2$\\
(.75*.22) cut at $\lambda$\\
(.75*.05) and put collar\\
\endSetLabels
\strut\AffixLabels{\includegraphics*[height=40mm]{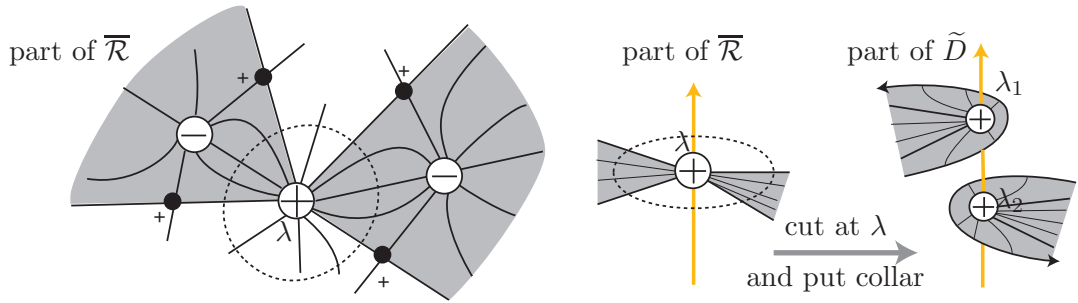}}
\end{center}
\caption{Case (i). Construction of $\widetilde D$ when $\lambda$ is a single positive elliptic point. The arrows depict the binding.}
\label{fig:neighbor-lambda}
\end{figure}
\begin{figure}[htbp]
\begin{center}
\SetLabels
(0*1)  part of $\overline{\mathcal R}$\\
(.14*.88) $\lambda$\\
(.18*.73) $w_1$\\
(.18*.94) $w_2$\\
(.0*.8) $v_1$\\
(0.25*.85)  $v_2$\\
(0.6*1)  part of $\widetilde D$\\
(.73*.86) $\lambda_1$\\
(.85*.86) $\lambda_2$\\
(.68*.7) $w_1'$\\
(.89*.7) $w_1''$\\
(.7*.96) $w_2''$\\
(.88*.95) $w_2'$\\
(.6*.8) $v_1$\\
(.96*.85)  $v_2$\\
(.4*.82) cut along $\lambda$\\
(.4*.76) and put collar\\
(0.68*0.39)  $v_2$\\
(.42*.39) $w_1''$\\
(.42*.46) $w_1'$\\
(.42*.59) $v_1$\\
(.68*.59) $w_2''$\\
(.68*.52) $w_2'$\\
(0.68*.05)  $v_2$\\
(.42*.05) $w_1''$\\
(.42*.1) $w_1'$\\
(.42*.24) $v_1$\\
(.68*.24) $w_2''$\\
(.68*.18) $w_2'$\\
(1.02*.19)  $v_2$\\
(.76*.19) $w_1''$\\
(.76*.25) $w_1'$\\
(.76*.38) $v_1$\\
(1.02*.38) $w_2''$\\
(1.02*.32) $w_2'$\\
(-.01*.45) $w_1$\\
(0.26*0.59)  $w_2$\\
(-.01*.59) $v_1$\\
(0.26*.45)  $v_2$\\
(-.01*.05) $w_1$\\
(0.26*0.18)  $w_2$\\
(-.01*.18) $v_1$\\
(0.26*0.05)  $v_2$\\
\endSetLabels
\strut\AffixLabels{\includegraphics*[height=100mm]{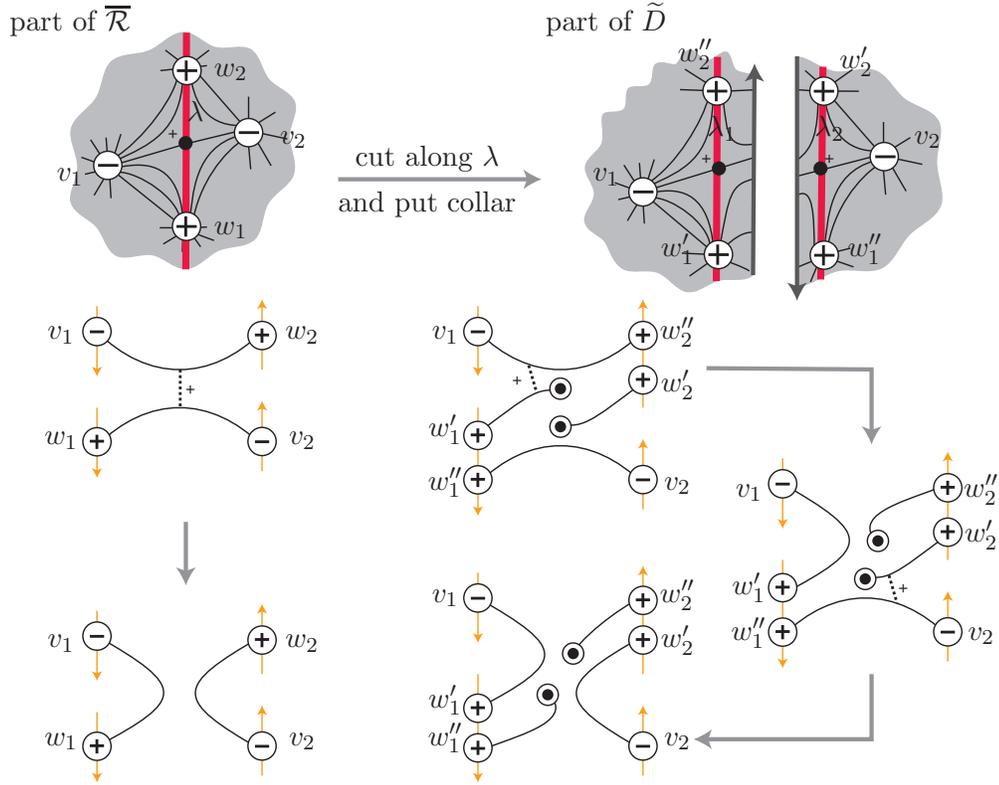}}
\end{center}
\caption{Case (ii). Construction of $\widetilde D$. The thin arrows represent part of the binding.}\label{fig:neighbor-lambda2}
\end{figure}
\begin{figure}[htbp]
\begin{center}
\SetLabels
(0.25*.95)  part of $\overline{\mathcal R}$\\
(0.6*.95) part of $\overline D$\\
(.96*.95)  part of $\widetilde D$\\
(.25*.75) $G_{--}$\\
(.6*.75) $G_{--}$\\
(.95*.75) $G_{--}$\\
(.1*.5) $\lambda$\\
(.45*.55) $\lambda_1$\\
(.45*.35) $\lambda_2$\\
(.78*.55) $\lambda_1$\\
(.78*.35) $\lambda_2$\\
(.32*.0) cut along $\lambda$\\
(.67*.0) put collar\\
\endSetLabels
\strut\AffixLabels{\includegraphics*[height=35mm]{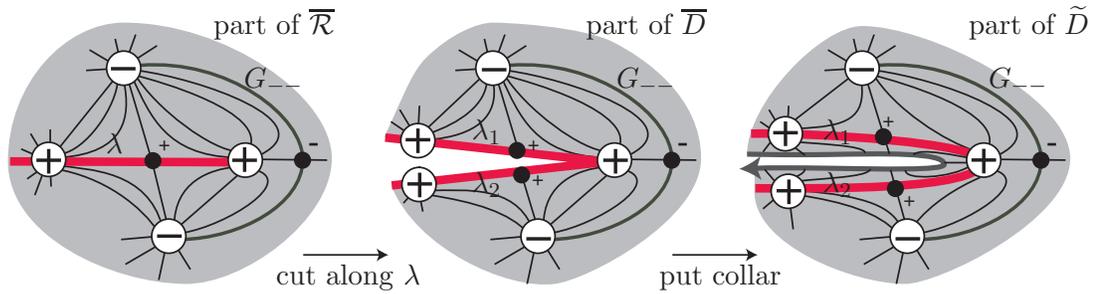}}
\end{center}
\caption{Case (ii) Construction of $\widetilde D$.}\label{fig:neighbor-lambda3}
\end{figure}
Cutting $\overline{\mathcal R}$ along $\lambda$ produces two copies of $\lambda$ which we denote by $\lambda_1$ and $\lambda_2$. 
Move $\lambda_2$ slightly away from $\lambda_1$ so that now $\partial \overline D$ is an embedded circle in $M$. 
We extend $\overline D$ by adding a collar neighborhood along $\partial \overline D$ so that the resulting surface, $\widetilde D$, is a disc embedded in $M$, its boundary $\partial \widetilde D$ is a positive transverse unknot, and the open book foliation  of the collar $\widetilde D\setminus \overline D$ has no singularities. 
Figure~\ref{fig:neighbor-lambda2} shows the change in open book foliation near $\lambda$ and corresponding movie presentations.

By the construction $\widetilde D$ satisfies all the requirements in Definition~\ref{def:trans-ot-disc}, so $\widetilde D$ is a transverse overtwisted disc. 
By Proposition~\ref{prop:overtwisted disc} we conclude that $\xi$ is overtwisted. 
\end{proof}

\begin{corollary}\label{cor:converse}
If a contact 3-manifold $(M_{(S, \phi)},\xi_{(S, \phi)})$ contains an overtwisted disc then $(S, \phi)$ contains a transverse overtwisted disc.
\end{corollary}

\begin{proof}
Let $\Delta \subset (M,\xi)$ be an overtwisted disc. We orient $\Delta$ so that the elliptic point of $\cF(D)$ has negative sign.
Since $\Delta$ is embedded and the boundary $L=\partial \Delta$ is a Legendrian knot, \cite[p.129]{etnyre} implies that we can take a collar neighborhood $\nu(\Delta)$ of $\Delta$ whose characteristic foliation $\cF(\nu(\Delta))$ is sketched in Figure~\ref{ot-disc2}. 
\begin{figure}[htbp]
\begin{center}
\SetLabels
(.8*.45) $L$\\
(.95*.45) $L^+$\\
\endSetLabels
\strut\AffixLabels{\includegraphics*[width=50mm]{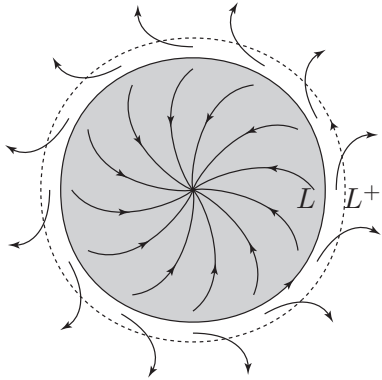}}
\caption{The characteristic foliation $\cF(\nu(\Delta))$ and a positive transverse push off $L^+$.}
\label{ot-disc2}
\end{center}
\end{figure}
Let $L^+ \subset \nu(\Delta)$ (dashed circle in Figure~\ref{ot-disc2}) be a positive transverse push off of $L$. 
Let $\Delta^+ \subset \nu(\Delta)$ be the disc bounded by $L^+$. 
Then $sl(L^+, [\Delta^+])=1$ and the Euler characteristic has $\chi(\Delta^+) = 1$. 
In particular,  $sl(L^+, [\Delta^+]) > -\chi(\Delta^+)$. 
By the proof of Theorem~\ref{theorem:BEinequality} we can find a transverse overtwisted disc.
\end{proof}

\section*{Acknowledgement}
The authors would like to thank Joan Birman, John Etnyre and the referees for numerous constructive comments. They also thank Marcos Ortiz for help with the English. 
The first author was supported by JSPS Research Fellowships for Young Scientists.  
The second author was partially supported by NSF grants DMS-0806492 and DMS-1206770.

\end{document}